\documentclass[oneside, 11pt]{amsart}

\usepackage[utf8]{inputenc}
\usepackage[T1]{fontenc}

\usepackage{hyperref}

\usepackage{xcolor}

\usepackage{amsfonts,amssymb,amsxtra,amstext}
\usepackage[english]{babel}
\usepackage{enumerate}
\usepackage{ulem}
\usepackage{mathtools}

\usepackage[margin=1in,top=1.2in, bottom=1.2in]{geometry}

\usepackage[
backend=biber,        
sorting=nyt,          
style=alphabetic,
]{biblatex}
 \makeatletter
 
 \@addtoreset{equation}{section}
 \makeatother

\newtheorem{theorem}{Theorem}[subsection]
\newtheorem{lemma}[theorem]{Lemma}
\newtheorem{corollary}[theorem]{Corollary}
\newtheorem{proposition}[theorem]{Proposition}
\newtheorem{remark}{Remark}[section]
\newtheorem{example}[remark]{Example}

\reversemarginpar 

\begin{document}

\title[Non-convergence to unstable equilibriums]{Non-convergence to
  unstable equilibriums for continuous-time and discrete-time stochastic
  processes}

\author{Olivier Raimond} \address{Laboratoire MODAL'X, Université Paris
  Nanterre, 200, avenue de la République - 92001 Nanterre cedex -
  France} \email{oraimond@parisnanterre.fr}
\urladdr{http://raimond.perso.math.cnrs.fr/}

\author{Pierre Tarrès}
\address{NYU-ECNU Institute of Mathematical Sciences at NYU Shanghai, China}
\email{tarres@nyu.edu}
\urladdr{}

%

%

\begin{abstract}
  We prove non-convergence theorems towards an unstable equilibrium (or
  a trap) for stochastic processes. The processes we consider are
  continuous-time or discrete-time processes and can be pertubations of
  the flow generated by a vector field. Our results extend previous
  results given for discrete-time processes by O. Brandière and M. Duflo
  in \cite{Brandiere96,Duflo1996}, by R. Pemantle in \cite{Pemantle90}
  and by P. Tarrès in \cite{Tarres00}. We correct and give a correct
  formulation to some theorems stated in
  \cite{Brandiere96,Duflo1996}. The method used to prove some of our
  theorems follow a method introduced by P. Tarrès in
  \cite{Tarres00}. Finally our non-convergence theorems are applied to
  give correct proofs of the non-convergence towards traps for the
  empirical measure of vertex reinforced random walks in \cite{Benaim13}
  and for non-backtracking vertex reinforced random walks in
  \cite{LR18}.
\end{abstract}

\maketitle

\tableofcontents

\section*{Notation}
The continuous-time stochastic processes we will consider will be
constructed on a filtered probability space
$(\Omega,\mathcal{F}, (\mathcal{F}_t)_{t\ge 0},\mathbb{P})$, with
$(\mathcal{F}_t)_{t\ge 0}$ a complete right-continuous filtration.

For $Z$ a càdlàg process in $\mathbb{R}^d$, we will set
$Z_{t-}:=\lim_{s\uparrow t} Z_s$ the left limit of $Z$ at $t>0$,
$\Delta Z_t=Z_t-Z_{t-}$ the jump of $Z$ at $t>0$, and $Z_{s,t}=Z_t-Z_s$
the increment of $Z$ between $s$ and $t$.

The total variation on an interval $(s,t]$ of a càdlàg real-valued
process $Z$ will be denoted by $V(Z,(s,t])$ and is defined by
\[
  V(Z,(s,t])=\sup_{\mathcal{P}}\sum_{i=1}^n |Z_{t_i}-Z_{t_{i-1}}|
\]
where the supremum is taken other all partitions
$\mathcal{P}=\{s=t_0<t_1<\dots<t_n=t\}$ of $(s,t]$.  Then
$t\mapsto V(Z,(0,t])$ is càdlàg and defines a measure on $(0,\infty)$.
If $V(Z,(0,t])<\infty$ for all $t$, then $Z$ will be said to have finite
variations.

The quadratic variation on an interval $(s,t]$ of a càdlàg
semimartingale $Z$ is denoted $[Z]_{s,t}$ and is defined as the limit
in probability of $\sum_{i=1}^n |Z_{t_i}-Z_{t_{i-1}}|^2$ as the mesh
of the partition $\mathcal{P}=\{s=t_0<t_1<\dots<t_n=t\}$ of $(s,t]$
goes to $0$. The quadratic variation of $Z$ on $(0,t]$ is denoted
$[Z]_t$ and we have that $[Z]_{s,t}=[Z]_t-[Z]_s$, if $s\le t$.  In the
case $Z$ has finite variations, we have that
$[Z]_{s,t}=\sum_{s<u\le t} (\Delta Z_u)^2$.  The quadratic covariation
process of two càdlàg semimartingales $Z$ and $Z'$ is defined by
$[Z,Z']:=\frac{1}{4}\left([Z+Z']-[Z-Z']\right)$.  Note that if $Z$ or
$Z'$ has finite variations, then
$[Z,Z']_{s,t}=\sum_{s<u\le t} (\Delta Z_u)(\Delta Z'_u)$. A
semimartingale \(Z\) is said purely discontinuous if its quadratic
variation \([Z]\) is a finite variation pure-jump process,
i.e. \([Z]_{s,t}=\sum_{s<u\le t} (\Delta Z_u)^2\). If \(Z\) and
\(Z'\) are two purely discontinuous semimartingales then
\([Z,Z']_{s,t}=\sum_{s<u\le t} (\Delta Z_u)(\Delta Z'_u)\).


If $M$ is a locally square integrable local martingale, its predictable
quadratic variation process is well defined and is the predictable
process, denoted $\langle M\rangle$, such that $[M]-\langle M\rangle$ is
a local martingale.  If $M$ and $M'$ are two locally square integrable
local martingales, then we define the predictable quadratic covariation
process
$\langle M,M'\rangle:=\frac{1}{4}\left(\langle M+M'\rangle-\langle
  M-M'\rangle\right)$ and we have that $[M,M']-\langle M,M'\rangle$ is a
local martingale.

Let $\langle\cdot,\cdot\rangle$ be an inner product on $\mathbb{R}^d$,
with associated norm denoted by $\|\cdot\|$. We will denote by
$x^1,\dots,x^d$ the coordinates of $x\in \mathbb{R}^d$ in an orthonormal
basis for this inner product.

If $Z$ is a càdlàg process in $\mathbb{R}^d$ such that for all $i$,
$Z^i$ is a semimartingale, we will set $V(Z,(s,t])=\sum_i V(Z^i,(s,t])$,
$[Z]_t=\sum_i [Z^i]_t$ and $[Z]_{s,t}=\sum_i [Z^i]_{s,t}$. And when
$Z^i$ is a local martingale for each $i$, we will set
$\langle Z\rangle_t=\sum_i \langle Z^i\rangle_t$ and
$\langle Z\rangle_{s,t}=\sum_i \langle Z^i\rangle_{s,t}$.

Let $\mathcal{N}$ be an open set of $\mathbb{R}^d$ and $\nu\in (0,1]$. A
function $f:\mathbb{R}^d\to \mathbb{R}^\delta$ will be said to belong
to $C^{1+\nu}(\mathcal{N})$ if the restriction of $f$ to $\mathcal{N}$
is $C^1$ with a $\nu$-hölderian differential $Df$. We will also say that
$f\in C^2(\mathcal{N})$ if $f$ restricted to $\mathcal{N}$ is $C^2$
(note that a function $f$ may belong to $C^{1+1}(\mathcal{N})$ but not
to $C^2(\mathcal{N})$).

If $\tau:[0,\infty)\to [0,\infty)$ is a right-continuous non-decreasing
function, its right-continuous inverse is the right-continuous
non-decreasing function $\tau^{-1}:[0,\infty)\to [0,\infty)$ defined by
$\tau^{-1}(t)=\inf\{s:\,\tau(s)>t\}$.

We will denote by $\mathcal{M}_d(\mathbb{R})$
(resp. $\mathcal{M}_d(\mathbb{C})$) the set of $d\times d$-matrices with
real-valued (resp. complex-valued) entries. A vector $x$ in
$\mathbb{R}^d$ will be viewed as a column
$x=\begin{pmatrix}x_1\\ \vdots\\ x_d\end{pmatrix}$.

We will use the following convention: $\frac{0}{0}=1$ and
$\frac{a}{0}=\infty$, if $a>0$.

\section{Introduction}
%

The question of how traps are avoided by stochastic approximation
algorithms has been studied by several authors (see
\cite{Pemantle90,Brandiere96,Benaim99,Tarres00,Benaim12} and the survey
\cite{Pemantle07}). In this paper, we address the same question to
continuous-time stochastic processes and apply these results to
discrete-time random processes (i.e. stochastic algorithms). This
permits to extend previous results given in \cite{Pemantle90} and in
\cite{Brandiere96}. A correction of the proof and a correct formulation
of \cite{Brandiere96}[Théorème 1] are also given (see Theorem
\ref{thm:th4} and remarks \ref{rk:duflo1} and \ref{rk:duflo2}).

Giving a correction to \cite{Brandiere96}[Théorème 1] had to be done:
This theorem has been applied by many different authors (here is a
non-exhaustive list:
\cite{zbMATH07107405,zbMATH06841011,zbMATH06727471,zbMATH05850271,zbMATH02100778,zbMATH02072449,zbMATH01441557,zbMATH01359623}
a more complete list is given in
\href{https://zbmath.org/?q=rf%3A898068}{zbMath} or in
\href{https://mathscinet-ams-org.ezproxy.math.cnrs.fr/mathscinet/search/publications.html?refcit=1387397&loc=headline}{Mathscinet}).

In the continuous-time setting the only known results are
\cite{Benaim05}[Theorem 2.26] (for a class of self-interacting
diffusions) and \cite{Nguyen20}[Theorem 5.12] (which is quite general
but that only allows a certain class of traps, i.e. a repulsive
equilibrium $x^*$ of a $C^1$ vector field $F$ such that
$DF(x^*)=\lambda I$, with $\lambda>0$ and $I$ the identity matrix).

In continuous-time the stochastic processes we consider take their
values in $\mathbb{R}^d$ and satisfy an equation such as
\[
  X_t=X_0+\int_0^t F_s \,\mathrm{d}s + M_t + R_t, \qquad \forall t\ge 0,
\]
where $M$ is a càdlàg martingale, $F$ is a progressively measurable
process and $R$ is a càdlàg adapted process with finite variations.  In
discrete-time the random sequences we consider satisfy
\[
  X_{n+1}-X_n=\gamma_{n+1} G_n + c_{n+1}(\varepsilon_{n+1}+r_{n+1}),
  \qquad \forall n\ge 0,
\]
where $\varepsilon_{n+1}$ is a martingale increment, $G$ and $r$ are
adapted sequences, and where $\gamma$ and $c$ are sequences of
non-negative real numbers.

We will especially be interested to the case where
$F_t=\gamma_t f(X_t)$ (with $\gamma_t\ge 0$) and $G_n=f(X_n)$, with
$f$ a $C^1$-function. A trap is an unstable equilibrium $x^*$ of $f$,
i.e. $f(x^*)=0$ and $Df(x^*)$ has an eigenvalue with a positive real
part. The goal is to give sufficient conditions ensuring the
a.s. non-convergence of $X$ towards an unstable equilibrium $x^*$.  A typical situation in discrete-time is the case
$\gamma_n=c_n=\frac{1}{n}$. If one supposes that
$\sum_n \|r_n\|^2<\infty$ and that $\varepsilon_n$ is bounded, then to
prove that $X_n$ cannot converge towards an unstable equilibrium is
not straightforward. Additional assumptions on the noise
$\varepsilon_n$ are required.

The main contributions of this paper are the following ones:
\begin{itemize}
\item A careful study of the continuous-time setting.
\item Theorem \ref{thm:th4d}, that gives (in the discrete-time setting)
  a correct formulation of \cite{Brandiere96}[Théorème 1] and extends
  \cite{Pemantle90}[Theorem 1] (see Remark \ref{rk:pemantle}).
\item Theorem \ref{thm:thhyp} that extends \cite{Benaim99}[Theorem
  9.1]: in \cite{Benaim99}[Theorem 9.1], it is supposed that $\gamma_n=c_n$
  are such that $\sum e^{-c/\gamma_n}<\infty$ for all $c>0$ and such
  that
  $\lim_{n\to\infty}\frac{\gamma_n^\alpha}{\sqrt{\sum_{k\ge
        n}\gamma_k^2}}$.  This is not required in Theorem
  \ref{thm:th4d}. Note that Theorem \ref{thm:thhyp} is essentially
  equivalent to \cite{Tarres00}[Theorem 2].
\item Giving stronger results for excitations of order $k\ge 2$ (see
  \cite{Brandiere96}[Corollary 5]), our results are stated as a remark
  at the end of Section \ref{sec:discrete}. These results can easily be
  obtained by passing through the continuous time.
\item In all our results, the assumptions on $\gamma$ are weaker than
  for the previously shown non-convergence theorems. The only results we
  have requiring a condition on $\gamma$ are \textcolor{black}{Proposition~\ref{thm:th22}}
  and Theorem~\ref{thm:th5}-\textit{(ii)} (in the continuous-time setting),
  Proposition~\ref{thm:th22n}, Theorem~\ref{thm:th4d}-\textit{(ii)}, Theorem
  \ref{thm:th5d} (in the discrete-time setting) and Theorem
  \ref{thm:thhyp}, whereas the previous results in the discrete-time
  setting all required (except in \cite{Tarres00}[Theorem 1], which
  deals only with one-dimensional stochastic algorithms) the following
  condition on $\gamma$: $\sum_n\gamma_n=\infty$ and $\gamma_n=O(c_n)$
  with $c_n$ such that $\sum_n c_n^2<\infty$. This condition is replaced
  in Theorem \ref{thm:th4d}-\textit{(ii)} by $\sum_n c_n^2<\infty$ and
  $\sum_{n>t}\gamma_n^2=0\big(\sqrt{\sum_{n>t} c_n^2}\big)$ as
  $t\to\infty$.
\end{itemize}

\bigskip The proof of Theorem \ref{thm:th2} given in this paper
essentially follows the proof of \cite{Tarres00}[Theorem 1]. In the note
\cite{Tarres00}, P. Tarrès proved a non-convergence theorem for
one-dimensional stochastic algorithms (note that to extend its proof to
multi-dimensional stochastic algorithms or to the continuous-time
setting is not straightforward).

\medskip Let us give the content of this paper in a little more detail:

\smallskip The continuous-time results are stated in Section
\ref{sec:thmct}.  We first suppose that $F_t$ is such that
$\langle X_{t-},F_t\rangle\ge 0$ and state the non-convergence
\textcolor{black}{Theorem~\ref{thm:th2} and
  Proposition~\ref{thm:th22}}. These results are proved in section
\ref{sec:secondTHM}.

We then suppose that $F_t=\gamma_t H_t X_t$, with $\gamma_t\ge 0$ and
$H_t$ a matrix that converges as $t\to\infty$ towards a repulsive
matrix $H$, i.e. a matrix such that its eigenvalues all have a
positive real part. Applying Theorem \ref{thm:th2} to this setting
allows us to prove \textcolor{black}{Corollary~\ref{thm:th3b}}, which
is proved in Section \ref{sec:nonCVrepulsive}.  We finally suppose
that $F_t=\gamma_t f(X_t)$, with $\gamma_t\ge 0$ and $f$ a $C^1$
vector field and state the main theorem in continuous time: Theorem
\ref{thm:th4}.  Sufficient conditions are given in Theorem
\ref{thm:th4} for the a.s. non-convergence of $X$ towards an unstable
equilibrium.  Theorem \ref{thm:th4} roughly says that if $x^*$ is
unstable and if \(X\) is ``sufficiently'' excited towards the unstable
directions of $x^*$ then $X$ cannot converge towards $x^*$.  The proof
of Theorem \ref{thm:th4} is given in Section \ref{sec:nonCVunstable}.

The last theorem in continuous-time is Theorem \ref{thm:th5}. This
theorem is proved in Section \ref{sec:lastnonCVthm}.  Theorem
\ref{thm:th5} is not implied by Theorem \ref{thm:th4} and can be useful
in many applications.  Again Theorem \ref{thm:th5} gives sufficient
conditions for the a.s. non convergence towards an unstable equilibrium.
This theorem permits to deal with situations where for example $f$ is
$C^{1+\nu}$ in the neighborhood of an unstable equilibrium and condition
\textit{(ii)} (i.e.  $\sum_{s\ge t}\|M_s^-\|^{1+\nu}=o(\alpha_t)$) of
Theorem \ref{thm:th4} is not satisfied.  But, giving additional
assumptions on $f$ and on $\gamma$ we first show that, a.s. on the event
$\{\lim_{t\to\infty}X_t=x^*\}$, the process $X$ is attracted
sufficiently fast towards the unstable manifold of $x^*$. This ensures
that $X$ is close after a large time to a stochastic process $Y^+$ for
which condition \textit{(i)} of Theorem \ref{thm:th4} is satisfied. This
permits to prove Theorem \ref{thm:th5}.  \medskip

The discrete-time results are stated in Section \ref{sec:thmdt}.  The
discrete-time version of Theorem~\ref{thm:th2},
\textcolor{black}{Proposition~\ref{thm:th22}},
\textcolor{black}{Corollary~\ref{thm:th3b}, Theorem~\ref{thm:th4} and
  Theorem~\ref{thm:th5}} are respectively Theorem~\ref{thm:th2n},
\textcolor{black}{Proposition~\ref{thm:th22n}},
\textcolor{black}{Corollary~\ref{thm:th3bd}, Theorem~\ref{thm:th4d} and
  Theorem~\ref{thm:th5d}.} \medskip

In Section \ref{subsec:nonCVhyperbolic}, sufficient conditions ensuring
the non-convergence towards a normally hyperbolic set (for example, an
``unstable'' orbit) are given in Theorem \ref{thm:thhyp}.  This theorem
is proved as a corollary of \textcolor{black}{Proposition~\ref{thm:th22n}}.

In Section \ref{sec:proof_ct} (respectively in Section
\ref{sec:discrete}) the proofs of the continuous-time non-convergence
theorems \ref{thm:th2n}, \textcolor{black}{\ref{thm:th22n}} and \ref{thm:th4d}
(respectively discrete-time non-convergence theorems) are given.

In Section \ref{sec:example}, we give two examples where Theorem
\ref{thm:th5d} can be used to give correct proofs of Theorem 3.9 in
\cite{Benaim05} and Theorem 3.27 (with additional assumptions that are
satisfied for NBVRRWs) in \cite{LR18}. Theorem \ref{thm:th5} could also
be used to prove Theorem 6.13 in \cite{Nguyen20}, and Theorem
\ref{thm:th4} could be used to prove Theorem 2.26 in \cite{Benaim05} and
Lemma 4.9 in \cite{Benaim02}.

\section{Presentation of the results}
\label{sec:results}

\subsection{Non-convergence results for continuous-time stochastic
  processes}\label{sec:thmct}
Let $\langle\cdot,\cdot\rangle$ be an inner product on $\mathbb{R}^d$
and let $X$ be a càdlàg stochastic process in $\mathbb{R}^d$ such that
\begin{equation}\label{eq:Z}
  X_t-X_0=\int_0^t F_s \,\mathrm{d}s + M_t+R_t
\end{equation}
where
\begin{itemize}
\item $F$ is a progressively measurable process,
\item For each $i$, $M^i$ is a càdlàg locally square integrable
  martingale,
\item For each $i$, $R^i$ is a càdlàg adapted process with finite
  variations.
\end{itemize}
\begin{remark}
  These last two assumptions do not depend on the inner product
  $\langle\cdot,\cdot\rangle$, i.e. if they are satisfied for one inner
  product, they are satisfied for all inner product.
\end{remark}


Let
\begin{itemize}
\item $\Gamma$ be an event, 
\item $\alpha:[0,\infty)\to (0,\infty)$ be a non-increasing càdlàg
  function such that $\lim_{t\to\infty} \alpha(t)=0$,
\item $\tau:[0,\infty)\to [0,\infty)$ be a non-decreasing càdlàg
  function such that $\tau(t) > t$ for all $t>0$ and such that
  $\lim_{n\to\infty}\tau^n(t)=\infty$ for all $t>0$ (with
  $(\tau^n)_{n\ge 0}$ defined by $\tau^{n+1}=\tau^n\circ\tau$ if
  $n\ge 0$ and $\tau^0(t)=t$ if $t\ge 0$).
\item \(\kappa:[0,1]\to[0,1]\) be a càdlàg increasing function such that
  for all \(p\in[0,1]\), \(\kappa(p)\geq{p}\).
\item \(E\) be a càdlàg adapted process taking its values in
  \([0,\infty]\) such that a.s. on \(\Gamma\)
  \begin{equation}\label{eq:DME}
    \frac{\mathbb{E}[\|\Delta{}M_S\|^21_{S<\infty}|\mathcal{F}_t]}{\kappa(\mathbb{P}[S<\infty|\mathcal{F}_t])}\leq{}E_{t}
  \end{equation}
  for all \(t\geq{0}\) and all stopping times \(S\) larger than \(t\).
\end{itemize}
In practice, the function \(\kappa\) and the process \(E\) will be
defined by \(\kappa(p)=p^{b}\) for some \(b\in(0,1]\) and by
\(E_{t}=k^{2}\alpha^{2}_{t}\) for some \(k\in(0,\infty)\).

Our first result, proved in Section \ref{sec:secondTHM}, is the
following one:
\begin{theorem}\label{thm:th2}
  Suppose that there are a constant \(\rho>0\) and a finite random
  variable $T_0$ such that a.s. on $\Gamma$,
  \begin{align}
    \label{hyp:F1}
    & \langle X_{t-},F_t\rangle \ge 0,\quad \hbox{ for almost all } t\ge T_0 \hbox{ such that } \|X_{t-}\|<\rho,\\
    \label{hyp:<M>1}
    & \alpha^2_{t}-\alpha^2_{\tau(t)}=O\left(\langle{M}\rangle_{t,\tau(t)}\right) \\
    \label{hyp:<M>2}
    &\langle{M}\rangle_{t,\infty}=O\left(\alpha^{2}_{t}\right), \\
    \label{hyp:calpha}
    &V(R,(t,\infty))=o(\alpha_t),\\
    \label{hyp:jumps}
    & E_{t}= O(\alpha^2_t). 
  \end{align} 
  Then it holds that
  $\mathbb{P}[\Gamma\cap\{\lim_{t\to\infty}X_t=0\}]=0$.
\end{theorem}

\begin{remark}
  The assumptions \eqref{hyp:F1}, \eqref{hyp:<M>1}, \eqref{hyp:<M>2},
  \eqref{hyp:calpha} and \eqref{hyp:jumps} do not depend on the inner
  product $\langle\cdot,\cdot\rangle$, i.e. if they are satisfied for
  one inner product, they are satisfied for all inner product.
\end{remark}

Assuming a stronger assumption on $F$ allows to have a weaker assumption
on $R$:
\begin{proposition}\label{thm:th22}
  Let $\gamma:[0,\infty)\to [0,\infty)$ be a measurable function such
  that $\lim_{t\to\infty}\gamma_t=0$ and such that
  $\int_t^\infty \gamma_s^2 ds=O(\alpha_t^2)$.  Suppose that
  $R_{t}=R'_{t}+\int_{0}^{t}r''_{s}\,\mathrm{d}s$, with $R'$ a càdlàg
  apdpted proces with finite variations and $r''$ a càdlàg adapted
  process, and that there are constants \(\rho>0\), $\nu>\frac12$ and
  a finite random variable $T_0$ such that a.s. on $\Gamma$,
  \eqref{hyp:<M>1}, \eqref{hyp:<M>2} and \eqref{hyp:jumps} are
  satisfied and that
  \begin{align}
    \label{hyp:F12}
    & \langle X_{t-},F_t\rangle \ge \gamma_t \|X_{t-}\|^2,\quad \hbox{ for almost all } t\ge T_0 \hbox{ such that } \|X_{t-}\|<\rho,  \\
    \label{hyp:calpha2}
    & V(R',(t,\infty))=o(\alpha_t),\\
    \label{hyp:calpha3}
    & r''_{t}=O(\gamma_{t}^{1+\nu}).
  \end{align} 
  Then it holds that
  $\mathbb{P}[\Gamma\cap\{\lim_{t\to\infty}X_t=0\}]=0$.
\end{proposition}

\begin{remark}\label{rk:th22}
  When $X$ is a non-negative one-dimensional process, it suffices to
  suppose in Proposition~\ref{thm:th22} that $R$ is such that
  $R_{t}\ge R'_{t}+\int_{0}^{t}r''_{s}\,\mathrm{d}s$, with $R'$ a
  càdlàg adapted process with finite variations and $r''$ a càdlàg
  adapted process such that a.s. on $\Gamma$, \eqref{hyp:calpha2} and
  \eqref{hyp:calpha3} are satisfied.
\end{remark}

\textcolor{black}{Theorem~\ref{thm:th2} and Proposition~\ref{thm:th22}}
will be proved in Section \ref{sec:secondTHM}.  In
Section~\ref{sec:nonCVrepulsive}, the following corollary of
Theorem~\ref{thm:th2} will be proved:

\begin{corollary}\label{thm:th3b}
  For each $t\ge 0$, let $\gamma_t$ be non-negative and $H_t$ be a
  $\mathcal{M}_d(\mathbb{R})$-valued random matrix.  Let
  $H\in \mathcal{M}_d(\mathbb{R})$ be a repulsive matrix, i.e. a matrix
  such that its eigenvalues all have a positive real part.  Suppose that
  there is a finite random variable $T_0$ such that a.s. on $\Gamma$,
  \begin{itemize}
  \item $F_t=\gamma_t H_tX_{t-}$ for almost all
    $t\ge T_0$; 
  \item $\lim_{t\to\infty}H_t=H$;
  \item conditions \eqref{hyp:<M>1}, \eqref{hyp:<M>2},
    \eqref{hyp:calpha} and \eqref{hyp:jumps} are satisfied.
  \end{itemize}
  Then it holds that
  $\mathbb{P}[\Gamma\cap\{\lim_{t\to\infty}X_t=0\}]=0$.
\end{corollary}

For the two remaining theorems, respectively proved in sections
\ref{sec:nonCVunstable} and \ref{sec:lastnonCVthm}, we let
$x^*\in \mathbb{R}^d$ and $f:\mathbb{R}^d\to\mathbb{R}^d$ be a function
such that $f$ is $C^{1}$ in a convex neighborhood $\mathcal{N}^*$ of
$x^*$. 
We will suppose that $x^*$ is an equilibrium for $f$, i.e. $f(x^*)=0$.
The equilibrium $x^*$ is said to be repulsive (resp. unstable) if the
eigenvalues of $Df(x^*)$ all have a positive real part (resp. if there
is an eigenvalue of $Df(x^*)$ having a positive real part).

The characteristic polynomial $\Pi^*$ of $H^*:=Df(x^*)$ can be written
as the product of two polynomials $\Pi^+$ and $\Pi^-$, such that all
roots of $\Pi^+$ (respectively of $\Pi^-$) all have a positive real part
(respectively all have a non-positive real part). Let $\delta^+$ and
$\delta^-$ be respectively the dimension of the kernel of $\Pi^+(H^*)$
and the dimension of the kernel of $\Pi^-(H^*)$.  Then there are an
invertible matrix $P\in\mathcal{M}_d(\mathbb{R})$, a matrix
$H^+\in\mathcal{M}_{\delta^+}(\mathbb{R})$ and a matrix
$H^-\in\mathcal{M}_{\delta^-}(\mathbb{R})$ such that
\begin{itemize}
\item
  $P^{-1} Df(x^*) P = \hbox{diag}[H^+,H^-]=\begin{pmatrix} H^+ & 0\\
    0&H^-
  \end{pmatrix}$;
\item The eigenvalues of $H^+$ all have a positive real part, i.e. $H^+$
  is repulsive;
\item The eigenvalues of $H^-$ all have a non-positive real part,
  i.e. $H^-$ is non-repulsive.
\end{itemize}
Note that if $x^*$ is unstable then $\delta^+\ge 1$ and if $x^*$ is
repulsive then $\delta^+=d$.  For $x\in\mathbb{R}^d$, set
$y:=P^{-1}x$. Let
$(y^+,y^-)\in\mathbb{R}^{\delta^+}\times\mathbb{R}^{\delta^-}$ be such
that $y=\begin{pmatrix} y^+\\y^-\end{pmatrix}$. Then
$P^{-1} Df(x^*) Py=\begin{pmatrix} H^+y^+\\H^-y^-\end{pmatrix}$.  Set
$M^\pm=(P^{-1}M)^\pm$ and $R^\pm=(P^{-1}R)^\pm$.

\begin{theorem}\label{thm:th4}
  Let $\gamma:[0,\infty)\to [0,\infty)$ be a measurable function and
  suppose that a.s. on $\Gamma$,
  \begin{itemize}
  \item $F_t=\gamma_t f(X_t)$, for almost all
    $t\ge 0$; 
  \item conditions \eqref{hyp:<M>1} and \eqref{hyp:<M>2} are satisfied
    by $M^+$, \eqref{hyp:calpha} and \eqref{hyp:jumps} are satisfied,
    and
    $\langle M^-\rangle_{t,\tau(t)}=O(\langle M^+\rangle_{t,\tau(t)})$.
  \end{itemize}
  Suppose also that one of the three following conditions is satisfied
  \begin{enumerate}[(i)]
  \item $x^*$ is repulsive;
  \item $x^*$ is unstable, $f\in C^{1+\nu}(\mathcal{N}^*)$ for some
    $\nu\in (0,1]$, $M$ is a purely discontinuous martingale and
    $\sum_{s\geq t}\|\Delta M^-_s\|^{1+\nu}=o(\alpha_t)$ on
    $\Gamma\cap\{\lim_{t\to\infty} X_t=x^*\}$;
  \item $x^*$ is unstable, $f\in C^2(\mathcal{N}^*)$ and
    $\sum_{s\geq t}\|\Delta M^-_s\|^{2}=o(\alpha_t)$ on
    $\Gamma\cap\{\lim_{t\to\infty} X_t=x^*\}$.
  \end{enumerate}
  Then it holds that
  $\mathbb{P}[\Gamma\cap\{\lim_{t\to\infty}X_t=x^*\}]=0$.
\end{theorem}

\begin{remark}
  Note that if a.s. on $\Gamma$, $M^+$ satisfies \eqref{hyp:<M>2} and
  $\langle M^-\rangle_{t,\tau(t)}=O(\langle M^+\rangle_{t,\tau(t)})$,
  then a.s. on $\Gamma$, $M^-$ satisfies \eqref{hyp:<M>2}. Indeed,
  a.s. on $\Gamma$,
  \begin{align*}
    \langle M^-\rangle_{t,\infty}
    &=\sum_{n=1}^\infty\langle{}M^-\rangle_{\tau^{n-1}(t),\tau^n(t)}=O\left(\sum_{n=1}^\infty\langle{}M^+\rangle_{\tau^{n-1}(t),\tau^n(t)}\right)\\
    &=O(\langle{}M^+\rangle_{t,\infty})=O(\alpha_t^2).
  \end{align*}
  Note also that if a.s. on $\Gamma$, $M^+$ satisfies \eqref{hyp:<M>1}
  and $\langle M^-\rangle_{t,\tau(t)}=0(\alpha^2_t-\alpha^2_{\tau(t)})$,
  then a.s. on $\Gamma$,
  $\langle M^-\rangle_{t,\tau(t)}=O(\langle M^+\rangle_{t,\tau(t)})$.
\end{remark}

\begin{remark}
  Theorem \ref{thm:th4} implies\footnote{Not completely true since the
    condition
    $\mathbb{E}\left[1_\Gamma \left(\int_t^\infty \| R_s\|
        ds\right)^2\right]=o(\alpha_t^2)$, given in \cite{Nguyen20},
    does not imply that a.s. on $\Gamma$,
    $\int_t^\infty \| R_s\| ds=o(\alpha_t)$.}  \cite{Nguyen20}[Theorem
  5.12], which states the non-convergence towards a repulsive
  equilibrium $x^*$ of $f$ such that $Df(x^*)=\lambda I$, with
  $\lambda>0$. Theorem \ref{thm:th4} can also be used to prove
  \cite{Benaim05}[Theorem 2.26] (in which it is proved that a.s. the
  empirical occupation measure of a self-interacting diffusion doesn't
  converge towards an unstable equilibrium of a vector field defined on
  the space of probability measures on a compact manifold).
\end{remark}

There are examples (see Section \ref{sec:example}) where the condition
$\sum_{s\ge t}\|\Delta{}M_s^-\|^{1+\nu}=o(\alpha_t)$ in \textit{(ii)}
of Theorem \ref{thm:th4} is not satisfied, but where \(\alpha_{t}\)
decreases towards $0$ sufficiently fast, so that \(X\) is strongly
attracted by the local unstable manifold in a neighborhood of
\(x^{*}\). The following theorem allows to deal with such situations.

When $x^*$ is an unstable equilibrium for $f$ such that $\delta^-\ge 1$,
we will denote by $\mu$ the largest real number such that every
eigenvalue of $H^-$ has a real part less or equal to $\mu$.  Then
$\mu\le 0$ and if $x^*$ is hyperbolic then one has that $\mu<0$.
\begin{theorem}\label{thm:th5}
  Let $\gamma:[0,\infty)\to [0,\infty)$ be a measurable function and
  suppose that $F_t=\gamma_t f(X_t)$, for almost all
  $t\ge 0$. For $t\ge 0$, set $m(t)=\int_0^t \gamma_s ds$ and suppose
  that
  \begin{enumerate}[{\itshape (i)}]
  \item $x^*$ is an hyperbolic unstable equilibrium of $f$ and that
    condition \eqref{hyp:h+r} of Section \ref{sec:Ft-assumption} is
    satisfied for some constant $\nu>0$. \footnote{A sufficient
      condition for \eqref{hyp:h+r} to be satisfied is that \(f\) is
      \(C^{1+\nu}\) in a neighborhood of \(x^{*}\) and that
      \(Df(x)v=0\) for all \(x\) in the local unstable manifold and
      \(v\in\Pi^{-}(H^{*})\), the vector space generated by the stable
      directions.}
  \item the function $\gamma$ is such that
    \begin{itemize}
    \item $\int_0^\infty \gamma_s ds = \infty$;
    \item
      $\lambda:=\limsup_{t\to\infty}\frac{\log(\alpha(t))}{m(t)} <0$;
    \item
      $\liminf_{t\to\infty} \frac{\log(\alpha(t))}{m(t)}>\beta (1+\nu)$,
    \end{itemize}
    where $\beta=\sup\{\lambda,\mu\}$.
  \item a.s. on $\Gamma$, conditions \eqref{hyp:<M>2} and
    \eqref{hyp:calpha} are satisfied by $M$ and $R$, condition
    \eqref{hyp:<M>1} is satisfied by $M^+$ and condition
    \eqref{hyp:jumps} is satisfied by $Y^+_t:=(P^{-1}X_t)^+$ (with
    \(M\) replaced by \((P^{-1}M)^{+}\) in~\eqref{eq:DME}).
  \end{enumerate}

  Then it holds that
  $\mathbb{P}[\Gamma\cap\{\lim_{t\to\infty}X_t=0\}]=0$.
\end{theorem}

\begin{remark}
  When $\lambda=\lim_{t\to\infty} \frac{\log(\alpha(t))}{m(t)}$, the
  condition
  $\liminf_{t\to\infty} \frac{\log(\alpha(t))}{m(t)}>\beta (1+\nu)$
  becomes $\lambda>\beta (1+\nu)$ which is equivalent to
  $\lambda> \mu (1+\nu)$ (since $\beta=\sup\{\lambda,\mu\}$ and
  $\nu>0$).  For example, when $\alpha_t=t^{-1/2}$ and
  $\gamma_t=(t+1)^{-1}$, we have $m(t)=\log(t+1)$ and
  $\lambda=-1/2$. And this condition becomes $\mu (1+\nu)<-1/2$.
\end{remark}

Conditions \textit{(i)} and \textit{(ii)} ensure that a.s. on
$\Gamma\cap\{\lim_{t\to\infty}X_t=x^*\}$, $X$ is attracted sufficiently
fast by the unstable manifold (see Lemma \ref{lem:beta-attract}).  Then
$X$ will be close after a large time to the stochastic process $Y^+$ for
which condition \textit{(i)} of Theorem \ref{thm:th4} will be satisfied.
This will allow to prove the almost sure non-convergence of $X$ towards
$x^*$ on $\Gamma$.

\subsection{Non-convergence results for discrete-time stochastic
  processes}
\label{sec:thmdt}

The question of avoiding traps for random processes is usually studied
for stochastic algorithms, i.e. for $(X_n)_{n\in \mathbb{N}}$ a random
sequence in $\mathbb{R}^d$, adapted to a filtration
$(\mathcal{F}_n)_{n\in\mathbb{N}}$, such that for all $n\ge 0$,
\begin{equation}\label{eq:Xn}
  X_{n+1}-X_n=\gamma_{n+1}G_n + c_{n+1}(\varepsilon_{n+1}+r_{n+1}),
\end{equation}
where $(\gamma_n)_{n\ge 0}$ and $(c_n)_{n\ge 0}$ are non-negative
deterministic sequences, $(\varepsilon_n)_{n\ge 0}$, $(r_n)_{n\ge 0}$
and $(G_n)_{n\ge 0}$ are adapted sequences such that for all $n\ge 0$,
$\mathbb{E}[\varepsilon_{n+1}|\mathcal{F}_n]=0$ and
$\mathbb{E}[\|\varepsilon_{n+1}\|^2|\mathcal{F}_n]<\infty$.  We will
assume that $c_n\ne 0$ infinitely often and that $\sum_n c_n^2<\infty$.

The results stated in this section are all corollaries of the theorems
stated in Section~\ref{sec:thmct}, these theorems being applied to the
time-continuous stochastic process \(X\) defined below. For $t\ge 0$
and $n=\lfloor t\rfloor$, set $\mathcal{F}_t=\mathcal{F}_{n}$. Then
$(\mathcal{F}_t)_{t\ge 0}$ is a complete right-continuous filtration.
Let us define the càdlàg processes \(F\), $M$, $R$ and $X$ such that
for all $n\ge 0$ and $t\in [n,n+1)$,
\begin{equation}
  F_t=\gamma_{n+1}G_n,\qquad M_t=\sum_{k=1}^n c_{k}\varepsilon_{k}, \qquad R_t=\sum_{k=1}^n c_{k} r_{k}
\end{equation}
and
\begin{align}
    & X_t=X_n+\int_n^t F_s ds.
\end{align}
Then \eqref{eq:Z} is satisfied by $X$, $F$, $M$ and $R$.

The details of the proofs of the following theorems are given in Section
\ref{sec:discrete}.

\medskip{} Theorem~\ref{thm:th2} implies the following theorem that
extends Theorem 1 in \cite{Tarres00}.
\begin{theorem}\label{thm:th2n}
  Let $\Gamma$ be an event and suppose that there are a finite random
  variable $T_0$, $a>2$ and $\rho>0$ such that a.s. on $\Gamma$,
  \begin{align}
    \label{hyp:Fn}
    & \langle X_{n}, G_n\rangle \ge 0,\quad \hbox{ for all $n\ge T_0$
      such that } \|X_{n}\|<\rho,\\
    \label{hyp:epsn} & 0<\liminf \mathbb{E}[\|\varepsilon_{n}\|^2|\mathcal{F}_{n-1}] \quad\hbox{ and }\quad \limsup \mathbb{E}[\|\varepsilon_{n}\|^a|\mathcal{F}_{n-1}]<\infty, \\
    \label{hyp:rn}& \sum_n \|r_n\|^2 < \infty.
  \end{align}
  Then it holds that
  $\mathbb{P}[\Gamma\cap\{\lim_{n\to\infty}X_n=0\}]=0$.
\end{theorem}

\textcolor{black}{Proposition~\ref{thm:th22}} implies the following proposition:
\begin{proposition}\label{thm:th22n}
  Let $\Gamma$ be an event.  Suppose that $\gamma_n=O(c_n)$, that
  \(r=r'+r''\) where \(r'\) is adapted and \(r''\) is predictable and
  that there are a finite random variable $T_0$, $a>2$, $\beta>0$,
  $\rho>0$ and $\nu>\frac12$ such that a.s. on $\Gamma$, condition
  \eqref{hyp:epsn} is satisfied and
  \begin{align}
    & \langle G_n,X_n\rangle \ge \beta \|X_n\|^2 \quad \hbox{ for all  }
      n\geq{T_{0}} \hbox{ such that }\quad \|X_n\|< \rho, \label{hyp:G22n}\\
    & \sum_n \|r'_n\|^2<\infty \quad \hbox{and} \quad r''_n=O\left(\frac{\gamma^{1+\nu}_n}{c_{n}}\right).\label{hyp:r22n}
  \end{align}
  Then it holds that
  $\mathbb{P}[\Gamma\cap\{\lim_{n\to\infty} X_n=0\}$.
\end{proposition}

\begin{remark}\label{rk:th22n}
  When $X$ is a non-negative one-dimensional process it suffices to
  suppose that $r$ is such that $r\ge r'+r''$, with $r'$ and $r''$
  such that a.s. on $\Gamma$, $\sum_n\|r'_n\|^2<\infty$ and
  $c_nr''_n=O(\gamma^{1+\nu}_n)$.
\end{remark}

\textcolor{black}{Corollary~\ref{thm:th3b} (or Theorem~\ref{thm:th2n})}
implies the following corollary that would correspond to
\cite{Brandiere96}[Proposition 4] (with different assumptions, but with
a correct proof).

\begin{corollary}\label{thm:th3bd}
  Let $\Gamma$ be an event. For each $n\ge 0$, let $H_n$ be an
  $\mathcal{F}_n$-measurable $\mathcal{M}_d(\mathbb{R})$-valued random
  matrix.  Let $H\in \mathcal{M}_d(\mathbb{R})$ be a repulsive matrix.
  Suppose that there are $a>2$ and a finite random variable $T_0$ such
  that a.s. on $\Gamma$,
  \begin{itemize}
  \item $G_n=H_nX_n$ for all $n\ge T_0$;
  \item $\lim_{n\to\infty}H_n=H$;
  \item conditions \eqref{hyp:epsn} and \eqref{hyp:rn} are satisfied.
  \end{itemize}
  Then it holds that
  $\mathbb{P}[\Gamma\cap\{\lim_{n\to\infty}X_n=0\}]=0$.
\end{corollary}

Let $x^*\in \mathbb{R}^d$ and $f:\mathbb{R}^d\to\mathbb{R}^d$ be a
function such that $f$ is $C^{1}$ in a convex neighborhood
$\mathcal{N}^*$ of $x^*$.  Suppose that for all $n\ge 0$, $G_n=f(X_n)$
and that $x^*$ is an unstable equilibrium for $f$.

Let us recall the notations from Section \ref{sec:nonCVunstable}: There
are an integer $\delta^+\in \{1,\dots,d\}$, an invertible matrix
$P\in\mathcal{M}_d(\mathbb{R})$, a repulsive matrix
$H^+\in\mathcal{M}_{\delta^+}(\mathbb{R})$ and a non-repulsive matrix
$H^-\in\mathcal{M}_{\delta^-}(\mathbb{R})$ (with $\delta^-:=d-\delta^+$)
such that
\begin{itemize}
\item $P^{-1} Df(x^*) P = \hbox{diag}[H^+,H^-]$;
\item The eigenvalues of $H^+$ all have a positive real part and the
  eigenvalues of $H^-$ all have a non-positive real part.
\end{itemize}
For $x\in\mathbb{R}^d$, set $y:=P^{-1}x$. There is
$(y^+,y^-)\in\mathbb{R}^{\delta^+}\times\mathbb{R}^{\delta^-}$ such that
$y=\begin{pmatrix} y^+\\y^-\end{pmatrix}$ and
$P^{-1} Df(x^*) Py=\begin{pmatrix} H^+y^+\\H^-y^-\end{pmatrix}$.  For
$n\ge 0$, set $\varepsilon_n^\pm=(P^{-1}\varepsilon_n)^\pm$ and
$r_n^\pm=(P^{-1}r_n)^\pm$.

\medskip The discrete time version of Theorem \ref{thm:th4} is the
following one. 

\begin{theorem}\label{thm:th4d}
  Let $\Gamma$ be an event and $x^*$ be an unstable equilibrium.
  Suppose that for some $a>2$, a.s. on $\Gamma$,
  \begin{align}
    \label{hyp:epsnd} & 0<\liminf \mathbb{E}[\|\varepsilon^+_{n}\|^2|\mathcal{F}_{n-1}] \quad\hbox{ and }\quad \limsup \mathbb{E}[\|\varepsilon_{n}\|^a|\mathcal{F}_{n-1}]<\infty. \\
    \label{hyp:rnd}& \sum_n \|r_n\|^2 < \infty.
  \end{align} 
  Set $\alpha(t)=\sqrt{\sum_{n> t} c_n^2}$ and suppose that one of the
  two following conditions is satisfied
  \begin{enumerate}[(i)]
  \item $x^*$ is repulsive;
  \item $\sum_{n>t}\gamma_{n}^2=O(\alpha(t))$ (which is satisfied if
    $\gamma_n=0(c_n)$) and there is $\nu\in (0,1]$ such that
    $f\in C^{1+\nu}(\mathcal{N}^*)$ and that
    $\sum_{n>t} c_n^{1+\nu} \|\varepsilon^-_n\|^{1+\nu}=o(\alpha(t))$ on
    $\Gamma\cap\{\lim X_n=x^*\}$.
  \end{enumerate}
  Then it holds that
  $\mathbb{P}[\Gamma\cap\{\lim_{n\to\infty}X_n=x^*\}]=0$.
\end{theorem}

\begin{remark}
  \label{rk:duflo1}
  Theorem \ref{thm:th2n} gives a correct formulation of
  \cite{Brandiere96}[Theorem 1]. But the setting is a bit different:
  \begin{itemize}
  \item In \cite{Brandiere96}[Theorem 1],
    \begin{itemize}
    \item it is supposed that $f\in C^{1+1}(\mathcal{N}^*)$
      (i.e. $\nu=1$) and there is a stronger assumption on the step of
      the algorithm: $\gamma_n=O(c_n)$ and $\sum \gamma_n=\infty$.
    \item Assumption~\eqref{hyp:epsn} is replaced by the
      non-equivalent assumption that
      \[
        0<\liminf \mathbb{E}[\|\varepsilon^+_{n}\||\mathcal{F}_{n-1}]
        \quad\hbox{ and }\quad \limsup
        \mathbb{E}[\|\varepsilon_{n}\|^2|\mathcal{F}_{n-1}]<\infty.
      \]
    \end{itemize}
  \item In Theorem~\ref{thm:th2n}, when \(\nu=1\), we have to assume
    that $\sum_{n>t} c_n^{2} \|\varepsilon^-_n\|^{2}=o(\alpha_t)$
    a.s. on $\Gamma$.  But if one looks carefully at the proof of
    \cite{Brandiere96}[Theorem 1], a similar condition has to be
    satisfied in \cite{Brandiere96}[Section I.4] (in order to obtain,
    using the notations of \cite{Brandiere96}, that
    $\sum\|\rho_n\|^2<\infty$).  Our result here does not permit to
    obtain \cite{Brandiere96}[Théorème 1].  \footnote{ When $\nu=1$,
      the hypothesis that
      $\limsup_{n\to\infty}
      \mathbb{E}[\|\varepsilon_n\|^2|\mathcal{F}_{n-1}]<\infty$
      a.s. on $\Gamma$ that is assumed in \cite{Brandiere96} is not
      far to imply that
      $\sum_{n>t} c_n^{2} \|\varepsilon_n\|^{2}=o(\alpha(t))$ a.s. on
      $\Gamma$: if one has that
      $\limsup_{n\to\infty} \mathbb{E}[\|\varepsilon_n\|^2]<\infty$
      then
      $\mathbb{E}[\sum_{n>t} c_n^{2}
      \|\varepsilon_n\|^{2}]=O(\alpha^2(t))=o(\alpha(t)).$ The
      condition
      $\sum_{n>t} c_n^{2} \|\varepsilon_n\|^{2}=o(\alpha(t))$ will be
      satisfied if
      $\limsup_{n\to\infty} \mathbb{E}[\|\varepsilon_n\|^4]<\infty$,
      since
      $\sum_{n>t} c_n^{2} \|\varepsilon_n\|^{2}\le \alpha(t)
      \sqrt{\sum_{n>t} c_n^{2} \|\varepsilon_n\|^{4}}=o(\alpha(t))$.
    }
  \end{itemize}
\end{remark}

\begin{remark}
  \label{rk:duflo2}
  There are several inaccuracies in the proof of
  \cite{Brandiere96}[Théorème 1] (also stated in \cite{Duflo1996}) :
  (with the notations of \cite{Brandiere96})
  \begin{enumerate}
  \item In the proof of \cite{Brandiere96}[Proposition 4], the
    application of \cite{Brandiere96}[Théorème A] done in
    \cite{Brandiere96}[page 407] requires that $R_n^1$ is
    adapted. This is not the case.
  \item In \cite{Brandiere96}[page 409]: it appears that to prove that
    $\sum_n \|\rho_{n+1}\|^2<\infty$ a.s. on $\Gamma(z^*)$, one needs to
    have that $\sum c_n^2\|\epsilon_{n+1}\|^4<\infty$ a.s. which is not
    necessarily satisfied under the assumption given in
    \cite{Brandiere96}[Théorème 1].
  \item In \cite{Brandiere96}[page 424]: there is a misuse in the
    application of an inequality of Burkholder.
  \end{enumerate}
\end{remark}

\begin{remark}
  Suppose that \(c_{n}=n^{-\gamma}\) with \(\gamma\in(1/2,1]\). Then
  \(\alpha(t)\sim c_{\gamma}t^{-\gamma+\frac{1}{2}}\) with
  \(c_{\gamma}=(2\gamma-1)^{-1/2}\) and
  \begin{itemize}
  \item if \(\varepsilon^{-}_{n}=O(1)\), then
    $\sum_{n>t} c_n^{1+\nu}
    \|\varepsilon^{-}_n\|^{1+\nu}=O(t^{-\gamma(1+\nu)+1})$ and
    $\sum_{n>t} c_n^{1+\nu} \|\varepsilon^{-}_n\|^{1+\nu}=o(\alpha(t))$ as
    soon as $\gamma\nu>1/2$;
  \item if \(1=O(\|\varepsilon_{n}^{-}\|)\), then
    $t^{-\gamma(1+\nu)+1}=O(\sum_{n>t} c_n^{1+\nu}
    \|\varepsilon^{-}_n\|^{1+\nu})$ and the condition
    $\sum_{n>t} c_n^{1+\nu}
    \|\varepsilon_{n}^{-}\|^{1+\nu}=o(\alpha_t)$ requires that
    $\gamma\nu>1/2$.
  \end{itemize}
  Note that the condition $\gamma\nu>1/2$ requires that \(\nu>1/2\).
\end{remark}

\begin{remark}
  \label{rk:pemantle}
  It is an easy exercise to check that Theorem \ref{thm:th4d} implies
  \cite{Pemantle90}[Theorem 1] (taking $\nu=1$ in \textit{(ii)},
  $\gamma_{n+1}=a_n$, $c_{n+1}=n^{-\gamma}$ and
  $\varepsilon_{n+1}=\xi_n/c_{n+1}$).
\end{remark}

We now give a discrete version of Theorem \ref{thm:th5}. Recall that
when $x^*$ is an unstable equilibrium for $f$ such that
$\delta^-\ge 1$, we denote by $\mu$ the largest real number such that
every eigenvalue of $H^-$ has a real part less or equal to $\mu$.
Recall alsi that \(\mu<0\) when $x^*$ is hyperbolic.

\begin{theorem}\label{thm:th5d}
  Let $\Gamma$ be an event and $x^*\in\mathbb{R}^d$.  Suppose that
  $G_n=f(X_n)$. Set $\alpha(t)=\sqrt{\sum_{n> t} c_n^2}$ and suppose
  that
  \begin{enumerate}[(i)]
  \item $x^*$ is an hyperbolic unstable equilibrium and \eqref{hyp:h+r}
    is satisfied for some $\nu>0$.
  \item $(\gamma_n)$ and $(c_n)$ satisfy:
    \begin{itemize}
    \item $\sum_{n>t}\gamma_{n}^2=O(\alpha(t))$;
    \item $\sum_n \gamma_n = \infty$;
    \item
      $\lambda:=\limsup_{t\to\infty}\frac{\log(\alpha(t))}{\sum_{k\le t}
        \gamma_k} <0$;
    \item
      $\liminf_{t\to\infty} \frac{\log(\alpha(t))}{\sum_{k\le t}
        \gamma_k}>\beta (1+\nu)$.
    \end{itemize}
    where \(\beta=\sup\{\lambda,\mu\}\).
  \item for some $a>2$, \eqref{hyp:epsnd} and \eqref{hyp:rnd} are
    satisfied a.s. on $\Gamma$.
  \end{enumerate}
  Then $\mathbb{P}[\Gamma\cap\{\lim_{n\to\infty}X_n=x^*\}]=0$.
\end{theorem}

\begin{remark}\label{rem:excitationk}
  Our framework easily allows excitations of order $k$: in the previous
  theorems, if one has that $\lim_{n\to\infty} \frac{c_{n+1}}{c_n}=1$,
  then conditions \eqref{hyp:epsn} and \eqref{hyp:epsnd} can be replaced
  by \eqref{hyp:epsnk} and \eqref{hyp:epsndk} with
  \begin{align}
    \label{hyp:epsnk}
    & 0<\liminf_{n\to\infty} \sum_{i=1}^k\mathbb{E}[\|\varepsilon_{n+i}\|^2|\mathcal{F}_{n+i-1}] \quad\hbox{ and }\quad \limsup_{n\to\infty} \mathbb{E}[\|\varepsilon_{n}\|^a|\mathcal{F}_{n-1}]<\infty\\
    \label{hyp:epsndk} & 0<\liminf_{n\to\infty} \sum_{i=1}^k\mathbb{E}[\|\varepsilon^+_{n+i}\|^2|\mathcal{F}_{n+i-1}] \quad\hbox{ and }\quad \limsup_{n\to\infty} \mathbb{E}[\|\varepsilon_{n}\|^a|\mathcal{F}_{n-1}]<\infty.
  \end{align}
  In the proof of these theorems, it suffices to take $\tau(t)=t+k$ and
  \(\mathcal{F}_{t}=\mathcal{F}_{n}\) if \(t\in[n,n+1)\).

  Conditions \eqref{hyp:epsn} and \eqref{hyp:epsnd} can also be replaced
  by \eqref{hyp:epsnkb} and \eqref{hyp:epsndkb} with
  \begin{align}
    \label{hyp:epsnkb}
    & 0<\liminf_{n\to\infty} \mathbb{E}[\sum_{i=1}^k\|\varepsilon_{n+i}\|^2|\mathcal{F}_{n}] \quad\hbox{ and }\quad \limsup_{n\to\infty} \mathbb{E}[\|\varepsilon_{n}\|^a|\mathcal{F}_{n-1}]<\infty\\
    \label{hyp:epsndkb} 
    & 0<\liminf_{n\to\infty} \mathbb{E}[\sum_{i=1}^k\|\varepsilon^+_{n+i}\|^2|\mathcal{F}_{n}] \quad\hbox{ and }\quad \limsup_{n\to\infty} \mathbb{E}[\|\varepsilon_{n}\|^a|\mathcal{F}_{n-1}]<\infty.
  \end{align}
  In the proof of these theorems, it suffices to take $\tau(t)=t+k$ and
  $\mathcal{F}_t=\mathcal{F}_{nk}$ if $t\in [nk,(n+1)k)$.
\end{remark}

\subsection{Non-convergence towards a normally hyperbolic set}
\label{subsec:nonCVhyperbolic}

Let $f:\mathbb{R}^d\to\mathbb{R}^d$ be a $C^{1+\nu}$ vector field, with
$\nu\in (1/2,1]$.  Let $S$ be a compact set invariant for the flow
$\Phi$ generated by $f$. Suppose that $S\subset M$, where $M$ is a
locally invariant submanifold of dimension $d-\delta$, with
$\delta\in\{1,...,d\}$.  Suppose that $\mathbb{R}^d=T_pM\oplus E^u_p$
for all $p\in S$ with: \textit{
  \begin{enumerate}[(i)]
  \item $p\mapsto E_p^u$ is a continuous mapping from $S$ into
    $G(\delta,d)$ the Grassmann manifold of $\delta$-dimensional
    subspaces of $\mathbb{R}^d$;
  \item $D\Phi_t(p)E_p^u=E_{\Phi_t(p)}^u$;
  \item there are constants $\lambda>0$ and $C>0$ such that, for all
    $p\in S$, $w\in E_p^u$ and $t\ge 0$,
    $\|D\Phi_t(p)w\|\ge Ce^{\lambda t}\|w\|$.
  \end{enumerate}}

Let $(X_n)_{n\in \mathbb{N}}$ be a sequence of random variables in
$\mathbb{R}^d$ adapted to a filtration
$(\mathcal{F}_n)_{n\in\mathbb{N}}$ satisfying \eqref{eq:Xn}.  Let
$\Gamma$ be an event.  We suppose that $\sum_n c_n^2<\infty$ and that
$\gamma_n=c_n$, and we set $\alpha_t=\sqrt{\sum_{n>t}c_n^2}$.  We also
suppose that there are a compact neighborhood $\mathcal{N}_0(S)$ of $S$,
$c>0$, $a>2$ and a finite random variable $T_0$ such that, a.s. on
$\Gamma$,
\begin{itemize}
\item $G_n=f(X_n)$ for all $n\ge T_0$,
\item $\sup_n \|\varepsilon_n\|<\infty$ and $\sum_n \|r_n\|^2<\infty$;
\item $\limsup \mathsf{{E}}[\|\varepsilon_n\|^a|\mathcal{F}_{n-1}]<\infty$,
\item
  $\mathbb{E}[\langle \varepsilon_{n+1},v\rangle^2|\mathcal{F}_n]\ge c
  1_{X_n\in\mathcal{N}_0(S)}$, for all $n\ge T_0$ and all unit vector
  $v\in\mathbb{R}^d$.
\end{itemize}

\begin{theorem} \label{thm:thhyp} Suppose the assumptions described
  above are satisfied. Then it holds that
  $\mathbb{P}[\Gamma\cap\{\lim_{n\to\infty}d(X_n,\Gamma)=0\}]$.
\end{theorem}
\begin{proof}
  Let $\eta$ be the non-negative Lipschitz function defined by Benaïm in
  \cite{Benaim99}[Proposition 9.5].  Without loss of generality, we will
  suppose that the constant $\beta$ in \cite{Benaim99}[Proposition 9.5]
  is equal to $1$ (if $\beta\ne 1$, it suffices to replace $\gamma_n$ by
  $\gamma_n/\beta$).  Note that we have slightly different notations to
  the ones in \cite{Benaim99}[Proposition 9.5]: $F$, $\alpha$, $\Gamma$
  and $S$ have respectively to be replaced by $f$, $\nu$, $S$ and $M$.

  We will apply \textcolor{black}{Proposition~\ref{thm:th22n} with
    Remark~\ref{rk:th22n}} to the sequence $\eta(X)$ and to the
  event
  $\Gamma_C:=\Gamma\cap\{\lim_{n\to\infty}d(X_n,S)=0\}\cap \{\sup_n
  \|\varepsilon_n+r_n\| \le C\}$ where $C$ is a positive constant.
  Note that there is a random variable $T\ge T_0$ such that a.s. on
  $\Gamma_C$, for all $n\ge T$ we have $X_n\in \mathcal{N}_0(S)$ and
  $\|r_{n+1}\|\le C$.  A.s. on $\Gamma_C$, we have that, for $n\ge T$,
  \begin{eqnarray*}
    \eta(X_{n+1})-\eta(X_n)&=& D\eta(X_{n})(X_{n+1}-X_n) + \{\eta(X_{n+1})-\eta(X_{n}) - D\eta(X_{n})(X_{n+1}-X_n)\}. 
  \end{eqnarray*}
  Let $G^\eta$, $\varepsilon^\eta$ and $r^\eta$ be the adapted sequences
  defined such that when $X_n\in \mathcal{N}_0(S)$,
  \begin{align*}
    & G^\eta_n=D\eta(X_{n})f(X_n),\quad \varepsilon^\eta_{n+1}=D\eta(X_{n})\varepsilon_{n+1} \\
    \text{ and } & r^\eta_{n+1}=D\eta(X_{n})r_{n+1}+\frac{1}{c_{n+1}}\{\eta(X_{n+1})-\eta(X_{n}) - D\eta(X_{n})(X_{n+1}-X_n)\}
  \end{align*}
  and such that when $X_n\notin \mathcal{N}_0(S)$,
  $G^\eta_n=\frac{\eta(X_{n+1})-\eta(X_n)}{\gamma_{n+1}}$,
  $\varepsilon^\eta_{n+1}=r^\eta_{n+1}=0$. Then $X$ satisfies for all
  $n\ge 0$,
  \[
    \eta(X_{n+1})-\eta(X_n)=\gamma_{n+1}G^\eta_n+c_{n+1}(\varepsilon^\eta_{n+1}+r^\eta_{n+1}).
  \]

  Assertion \textit{(iii)} of \cite{Benaim99}[Proposition 9.5] ensures
  that there is a compact neighborhood
  $\mathcal{N}(S)\subset \mathcal{N}_0(S)$ of $S$ and a finite constant
  $k$ such that, a.s. on $\Gamma_C$, for all $n\ge T$,
  \[
    \eta(X_{n+1})-\eta(X_{n}) - D\eta(X_{n})(X_{n+1}-X_n)\ge -k
    \gamma_{n+1}^{1+\nu}.
  \]
  Let $r'$ and $r''$ be the adapted sequences defined such that when
  $X_n\in \mathcal{N}(S)$, $r'_{n+1}=D\eta(X_{n})r_{n+1}$ and
  $r''_{n+1}=-\frac{k}{c_{n+1}}\gamma_{n+1}^{1+\nu}$ and such that
  when $X_n\notin \mathcal{N}(S)$, $r'_{n+1}=r''_{n+1}=0$.  Then,
  a.s. on $\Gamma_C$, for all $n\ge T$ we have
  $r^\eta_{n+1}\ge r'_{n+1}+r''_{n+1}$.

  Assertion \textit{(vi)} of \cite{Benaim99}[Proposition 9.5] ensures
  that, when $X_n\in\mathcal{N}(S)$,
  $G^\eta_n=D\eta(X_{n})f(X_n)\ge \eta(X_n)$ and so, a.s. on $\Gamma_C$,
  \eqref{hyp:G22n} is satisfied for all $n\ge T$.  One also easily
  checks that a.s. on $\Gamma_C$, \eqref{hyp:r22n} is satisfied,
  i.e. $\sum_n |r'_n|^2<\infty$ and $c_nr''_n=O(\gamma_n^{1+\nu})$.

  In order to apply \textcolor{black}{Proposition~\ref{thm:th22n}}, it remains to check that
  condition \ref{hyp:epsn} is satisfied by $\varepsilon^\eta$ a.s. on
  $\Gamma_C$.  It is straightforward to check that, a.s. on $\Gamma_C$,
  $\limsup \mathbb{E}[|\varepsilon^\eta_{n+1}|^a|\mathcal{F}_{n}]<\infty$.
  The fact that a.s. on $\Gamma_C$,
  $\liminf \mathbb{E}[|\varepsilon_{n+1}^\eta|^2|\mathcal{F}_{n}]>0$,
  follows from the fact that a.s. on $\Gamma_C$, for all $n\ge T$,
  \[
    \mathbb{E}[|\varepsilon_{n+1}^\eta|^2|\mathcal{F}_{n}] =
    \mathbb{E}[|\langle
    D\eta(X_n),\varepsilon_{n+1}^\eta\rangle|^2|\mathcal{F}_n]\ge
    c\|D\eta(X_n)\|^2 1_{X_n\in\mathcal{N}(S)}.
  \]
  Therefore
  $\mathbb{P}[\Gamma\cap\{\lim_{n\to\infty}d(X_n,\Gamma)=0\}]=\lim_{C\to\infty}\mathbb{P}[\Gamma_C\cap
  \{\lim_{n\to\infty}\eta(X_n)=0\}]=0$.
\end{proof}

Theorem \ref{thm:thhyp} is essentially equivalent to
\cite{Tarres00}[Theorem 2] and extends \cite{Benaim99}[Theorem 9.1]:
there is an additional term $r_n$ and
\begin{enumerate}
\item in \cite{Benaim99}[Theorem 9.1], it is assumed that
  $\lim_{n\to\infty}\frac{\gamma_{n+1}^\nu}{\sqrt{\sum_{i=n+1}^m\gamma_i^2}}=0$
  and in Theorem \ref{thm:thhyp} it is just assumed that
  $\sum_n c_n^2<\infty$.
\item the assumption on the noise $\varepsilon$ is weaker in Theorem
  \ref{thm:thhyp} than in \cite{Benaim99}[Theorem 9.1]. Indeed,
  condition \textit{(i)} in \cite{Benaim99}[Theorem 9.1] (i.e. for all
  unit vector $v$,
  $E[\langle \varepsilon_{n+1},v\rangle_+|\mathcal{F}_n]\ge b 1_{X_n\in S}$,
  with $b>0$) implies that for all unit vector $v$,
  $E[\langle \varepsilon_{n+1},v\rangle^2|\mathcal{F}_n]^{1/2}\ge E[|\langle
  \varepsilon_{n+1},v\rangle||\mathcal{F}_n]=2E[\langle
  \varepsilon_{n+1},v\rangle_+|\mathcal{F}_n]\ge 2b 1_{X_n\in S}$.
\end{enumerate}
In \cite{Benaim99}[Theorem 9.1], it is only assumed that $\nu\in (0,1]$,
but, as it is noticed in \cite{Tarres00}, the proof of this theorem
requires $\nu>1/2$.


\section{Proofs of the non-convergence theorems for continuous-time
  processes}
\label{sec:proof_ct}

\subsection{Proofs of Theorem~\ref{thm:th2} and
  \textcolor{black}{Proposition~\ref{thm:th22}}}
\label{sec:secondTHM}
In this section, we let $X$ be a càdlàg process in $\mathbb{R}^d$
satisfying \eqref{eq:Z} and we let \(\Gamma\),
\(\alpha\), \(\tau\), \(\kappa\) and \(E\) be the event, functions and
process introduced just before the statement of
Theorem~\ref{thm:th2}. We suppose that there are a finite random variable $T_0$ and \(\rho>0\)
such that a.s. on $\Gamma$,
\eqref{hyp:F1}~\eqref{hyp:<M>1}~\eqref{hyp:<M>2}~\eqref{hyp:calpha}
and \eqref{hyp:jumps} are satisfied.

Before proving Theorem \ref{thm:th2}, i.e. under the assumptions given
above it holds that
$\mathbb{P}[\Gamma\cap\{\lim_{n\to\infty}X_n=0\}]=0$, we first give a
sufficient condition ensuring that \eqref{hyp:jumps} holds.
\begin{lemma} \label{lem:E_t}
  If there is $a>2$ and $t_0>0$ such that a.s. on \(\Gamma\), for all
  $s\ge t\ge t_0$, one has that
  \begin{equation}
    \mathbb{E}[\|\Delta M_s\|^a|\mathcal{F}_t]^{2/a} \le -k^2\Delta\alpha^2_s,
  \end{equation}
  then \eqref{hyp:jumps} holds a.s. on \(\Gamma\),
  with \(\kappa(p)=p^{b}\) with $b=1-\frac{2}{a}$ and with \(E\) the
  càdlàg process defined by \(E_{t}=k^{2}\alpha^{2}_{t}\) if
  \(t\geq{}t_{0}\) and \(E_{t}=\infty\) if \(t<t_{0}\). 
\end{lemma}
\begin{proof}
  Let $t\ge t_0$ and $S$ be a stopping time larger than $t$.  Then
  (setting $b=1-\frac{2}{a}$)
  \begin{align*}
    \mathbb{E}[\|\Delta M_S\|^21_{S<\infty}|\mathcal{F}_t]
    &\le \; \sum_{s>t} \mathbb{E}[\|\Delta M_s\|^21_{S=s}|\mathcal{F}_t]\\
    &\le \; \sum_{s>t} \mathbb{E}[\|\Delta M_s\|^a|\mathcal{F}_t]^{2/a} \mathbb{P}[S=s|\mathcal{F}_t]^b \\
    &\le \; -\sum_{s>t} k^2(\Delta\alpha^2_s) \mathbb{P}[S<\infty|\mathcal{F}_t]^b \\
    &\le \; k^2\alpha^2_t \mathbb{P}[S<\infty|\mathcal{F}_t]^b,
  \end{align*}
  Where the two last inequalities hold a.s. on \(\Gamma\).  We thus
  have that, a.s. on \(\Gamma\), for all stopping time \(S\) larger
  than \(t\geq{0}\),
  \(\frac{\mathbb{E}[\|\Delta
    M_S\|^21_{S<\infty}|\mathcal{F}_t]}{\kappa(\mathbb{P}[S<\infty|\mathcal{F}_t])}\leq{E_{t}}\)
  and \(E_{t}=O(\alpha^{2}_{t})\) as \(t\to\infty\). 
\end{proof}

\subsubsection{Simplification of the hypotheses.}\label{sec:Simplification}
We start by simplifying the hypotheses. Note that a.s.,
$\Gamma=\cap_{\varepsilon\in (0,1]}\cup_{(t_0,k)\in\mathbb{N}_*^2}
\Gamma_{t_0,k,\varepsilon}$, where for $(t_0,k)\in \mathbb{N}_*^2$ and
$\varepsilon\in (0,1]$, $\Gamma_{t_0,k,\varepsilon}$ is the set of all
$\omega\in \Gamma$ for which $T_0\le t_0$ and such that for all
$t\ge t_0$,
\begin{align}
  \label{hyp:F}
  & \langle X_{t-},F_t\rangle 1_{\{\|X_{t-}\|\le \rho\}} \ge 0,\\
  & \langle M\rangle_{t,\tau(t)}\geq k^{-2}(\alpha^2_{t}-\alpha^2_{\tau(t)}),\\
  & \langle M\rangle_{t,\infty}\leq k^2\alpha^2_t,\\
  \label{hyp:kt0calpha}
  & V(R,(t,\infty))\leq \varepsilon\alpha_t, \\
  \label{hyp:kt_ojumps0}
  & E_{t}\leq\frac{k^2}{4}\alpha^2_t. 
\end{align} 
It is thus sufficient to prove Theorem \ref{thm:th2} with $\Gamma$
replaced with $\Gamma_{t_0,k,\varepsilon}$ for all $(t_0,k,\varepsilon)$
with $t_0$ sufficiently large and $\varepsilon$ sufficiently small.  Note
that on $\Gamma_{t_0,k,\varepsilon}$, for $t\ge t_0$, we have (using
that
$\langle M\rangle_{t,\infty}=\sum_{n\ge 1} \langle
M\rangle_{\tau^{n-1}(t),\tau^n(t)}\ge k^{-2}\sum_{n\ge 1}
\big(\alpha^2_{\tau^{n-1}(t)}-\alpha^2_{\tau^{n}(t)}\big)$)
\begin{equation}
  \label{hyp:kt0<M>}
  k^{-2}\alpha^2_t\le\langle M\rangle_{t,\infty}\leq k^2\alpha^2_t.
\end{equation}
Note also that \eqref{hyp:kt0calpha} and \eqref{hyp:kt_ojumps0} imply that,
for $\varepsilon<k/2$ and $t\ge t_0$,
\begin{equation}
  \label{hyp:kt_ojumps}
  \mathbb{E}[\|\Delta X_S\|^21_{S<\infty}|\mathcal{F}_t] \le k^2\alpha^2_t\times\kappa(\mathbb{P}[S<\infty|\mathcal{F}_t]), \hbox{ for any stopping times $S$ larger than $t$}.
\end{equation}
For the rest of this section, we will suppose that
$\Gamma=\Gamma_{t_0,k,\varepsilon}$, with $(t_0,k)\in \mathbb{N}_*^2$
and $\varepsilon<k/2$. We will also set \(V_{t}=V(R,(0,t])\) (then
\(V(R,(s,t])=V_{t}-V_{s}=V_{s,t}\)).

Let $T$ be the first time $t\ge \tau(t_0)>t_0$ such that one the
following items hold
\begin{enumerate}[{\itshape (i)}]
\item $\langle X_{t-},F_t\rangle 1_{\{\|X_{t-}\|<\rho\}} < 0 $,
\item
  $\langle M\rangle_{t}-\langle M\rangle_{\tau^{-1}(t)}<
  k^{-2}(\alpha^2_{\tau^{-1}(t)}-\alpha^2_t)$,
\item
  $\sup_{s\in [t_0,t]}\frac{\langle M\rangle_t-\langle
    M\rangle_s}{\alpha^2_s} > k^2$,
\item $\sup_{s\in [t_0,t]}\frac{V_t-V_s}{\alpha_s} > \varepsilon$,
\item \(E_{t}>\frac{k^2}{4}\alpha^2_t\). 
\end{enumerate}
Then $T$ is a stopping time and $T=\infty$ a.s. on $\Gamma$.  Possibly
extending the probability space, define new processes $X'$, $F'$, $M'$
and $R'$ such that if $t< T$, $X'_t=X_t$, $F'_t=F_t$, $M'_t=M_t$ and
$R'_t=R_t$, and such that for $t\ge T$, $F'_t=R'_t=0$ and
$M'_t-M_T=\tilde{M}_t-\tilde{M}_T$, where $\tilde{M}$ is a
martingale\footnote{One can take $\tilde{M}_t=(B_{A_t},0,\dots,0)$,
  where $B$ is an independent standard Brownian motion and
  $A_t=\alpha^2_0-\alpha^2_t$. } in $\mathbb{R}^d$ such that
$\langle \tilde{M}\rangle_t=\alpha^2_0-\alpha^2_t$.  Then, it can be
checked that $X'$, $F'$, $M'$ and $R'$ satisfy a.s. for all
$t\ge \tau(t_0)$ conditions \eqref{hyp:F}, \eqref{hyp:kt0<M>},
\eqref{hyp:kt0calpha} and \eqref{hyp:kt_ojumps} and that $X'=X$ on
$\Gamma$.  This implies that
$\mathbb{P}[\Gamma\cap\{\lim_{t\to\infty}
X_t=0\}]=\mathbb{P}[\Gamma\cap\{\lim_{t\to\infty} X'_t=0\}]\le
\mathbb{P}[\lim_{t\to\infty} X'_t=0]$.  Therefore to prove Theorem
\ref{thm:th2} it suffices to prove that
$\mathbb{P}[\lim_{t\to\infty} X'_t=0]=0$.

\subsubsection{A non-convergence proposition.}
We suppose in this subsection that there are $t_0\ge 0$, $k\ge 1$,
$\rho>0$, $\varepsilon>0$ and an increasing function
\(\kappa:[0,1]\to[0,1]\) satisfying \(\kappa(p)\geq{p}\) for all
\(p\in[0,1]\), 
such that almost surely, for all $t>t_0$, conditions \eqref{hyp:F},
\eqref{hyp:kt0<M>}, \eqref{hyp:kt0calpha} and \eqref{hyp:kt_ojumps} are
satisfied, i.e.
\begin{itemize}
\item $\langle X_{t-},F_t\rangle 1_{\{\|X_{t-}\|\le \rho\}} \ge 0$,
\item
  $k^{-2}\alpha^2_t\le\langle M\rangle_{t,\infty} \leq
  k^2\alpha^2_t$,
\item $V_\infty - V_t\le \varepsilon\alpha_t$,
\item
  $\mathbb{E}[\|\Delta X_S\|^21_{S<\infty}|\mathcal{F}_t] \le
  k^2\alpha^2_t\times\kappa(\mathbb{P}[S<\infty|\mathcal{F}_t])$, for any
  stopping times $S$ larger than $t$. 
\end{itemize}
Note that condition \eqref{hyp:kt0<M>} ensures that for each $i$, $M^i$
is a square integrable martingale and that $[M]-\langle M\rangle$ is a
martingale.

In this subsection, we will prove the following proposition, which
will allow us to conclude the proof of Theorem \ref{thm:th2}.
\begin{proposition}\label{prop:nonCV}
  If { $\varepsilon\le \frac{1}{2^7k^3\sqrt{2d}}$ and
    $\alpha_{t_0}\le \frac{\rho}{2^3 k\sqrt{2d}}$}, then
  $\mathbb{P}[\lim_{t\to\infty} X_t=0]=0$.
\end{proposition}

To prove this proposition, we follow (and adapt to our framework)
\cite{Tarres00}.  This proof is a consequence of two lemmas.

For $t\ge t_0$, let $S_t=\inf\{s\ge t:\, \|X_s\|\ge L\alpha_s\}$, where
$L$ is a positive constant we will fix later on.

\begin{lemma}\label{lem:lem1} 
  If $(2L+\varepsilon)\varepsilon\le (2k)^{-2}$ and
  $\rho\ge L\alpha_{t_0}$, then there is $p>0$ such that for all
  $t\ge t_0$, $\mathbb{P}[S_t<\infty|\mathcal{F}_t]\ge p$.
\end{lemma}
\begin{proof}
  Let us fix $t\ge t_0$.  For $s\in (t,S_t)$, we have $\|X_{s-}\|<\rho$ and
  for $s>t$,
  \begin{align}
    \label{eq:normX2}\|X_{s\wedge S_t}\|^2
    =&\; \|X_t\|^2 +2\int_t^{s\wedge S_t} \langle X_{u-},dX_u\rangle + [X]_{s\wedge S_t}-[X]_t.
  \end{align}
  One has $[X]=[M+R]$,
  \begin{align}
    \label{eq:[X]}[X]_{s\wedge S_t}-[X]_t & \;\ge \frac{[M]_{s\wedge S_t}-[M]_t}{2}-([R]_{s\wedge S_t}-[R]_t)
  \end{align}
  and
  \begin{align}
    \nonumber\int_t^{s\wedge S_t} \langle X_{u-},dX_u\rangle &\;  = \int_t^{s\wedge S_t} \langle X_{u-},F_u\rangle du + \int_t^{s\wedge S_t} \langle X_{u-},dM_u\rangle
                                                               + \int_t^{s\wedge S_t} \langle X_{u-},dR_u\rangle\\
    \label{eq:intXX}&\;\ge \int_t^{s\wedge S_t} \langle X_{u-},dM_u\rangle
                      + \int_t^{s\wedge S_t} \langle X_{u-},dR_u\rangle.
  \end{align}  
  For all $t$, let $\bar{M}^t$ be the martingale stopped at $S_t$
  defined by
  \begin{align}
    \label{eq:defMbar}\bar{M}^t_s =&\; 2\int_t^{s\wedge S_t} \langle X_{u-},dM_u\rangle 
                                     +\frac{1}{2}\int_t^{s\wedge S_t} d([M]_u-\langle M\rangle_u).
  \end{align}
  Using \eqref{eq:[X]}, \eqref{eq:intXX} and \eqref{eq:defMbar},
  \eqref{eq:normX2} implies that
  \begin{align*}
    \|X_{s\wedge S_t}\|^2
    \ge &\;  \bar{M}^t_s
          + 2\int_t^{s\wedge S_t} \langle X_{u-},dR_u\rangle
          + \frac{1}{2}\big(\langle M\rangle_{s\wedge S_t}-\langle M\rangle_t\big)-\big([R]_{s\wedge S_t}-[R]_t\big).
  \end{align*}

  We have that
  \begin{align*}
    \left|\int_t^{s\wedge S_t} \langle X_{u-},dR_u\rangle \right|
    \le&\; \int_t^{s\wedge S_t} \|X_{u-}\|dV_u \le  L\int_t^\infty \alpha_u dV_u\\
    \le &\; -L\int_t^\infty (V_u-V_t)d\alpha_u
          \le L\varepsilon\alpha^2_t
  \end{align*}
  and
  \begin{align*}
    &[R]_{s\wedge S_t}-[R]_t \le [R]_\infty-[R]_t \le \sup_{s\ge t} \|\Delta R_s\| \times \big( V_\infty - V_t\big) \le \varepsilon^2 \alpha^2_t.
  \end{align*}
  Therefore,
  \begin{align*}
    \|X_{s\wedge S_t}\|^2
    \ge &\;  \bar{M}^t_s
          -(2L+\varepsilon)\varepsilon\alpha^2_t
          + \frac{1}{2}\big(\langle M\rangle_{s\wedge S_t}-\langle M\rangle_t\big).
  \end{align*}

  Using that $(2L+\varepsilon)\varepsilon\le \frac{1}{4k^2}$, the
  martingale property yields
  \begin{align*}
    \liminf_{s\to\infty}\mathbb{E}[\|X_{s\wedge S_t}\|^2|\mathcal{F}_t]
    &\ge 
      -\frac{\alpha^2_t}{4k^2} + \frac{1}{2}\liminf_{s\to\infty} \mathbb{E}[1_{\{S_t> s\}}(\langle M\rangle_s-\langle M\rangle_t)|\mathcal{F}_t]\\
    &\ge 
      -\frac{\alpha^2_t}{4k^2} + \frac{1}{2}\mathbb{E}[1_{\{S_t=\infty\}}\langle M\rangle_{t,\infty}|\mathcal{F}_t]\\
    &\ge 
      -\frac{\alpha^2_t}{4k^2} + \frac{\alpha^2_t}{2k^2}
      \mathbb{P}[S_t=\infty|\mathcal{F}_t]\\
    &\ge 
      \frac{\alpha^2_t}{4k^2} - \frac{\alpha^2_t}{2k^2} \mathbb{P}[S_t<\infty|\mathcal{F}_t]
  \end{align*}
  We also have
  \begin{align*}
    \mathbb{E}[\|X_{s\wedge S_t}\|^2|\mathcal{F}_t]
    &=\mathbb{E}[\|X_{s\wedge S_t}\|^21_{\{S_t\le s\}}|\mathcal{F}_t]
      + \mathbb{E}[\|X_{s\wedge S_t}\|^21_{\{S_t> s\}}|\mathcal{F}_t]\\
    &\le \mathbb{E}[(L\alpha_t+\|\Delta X_{S_t}\|)^21_{\{S_t\le s\}}|\mathcal{F}_t] + L^2\alpha^2_s \mathbb{P}[S_t>s|\mathcal{F}_t] 
  \end{align*}
  and therefore
  \begin{align*}
    \liminf_{s\to\infty}\mathbb{E}[\|X_{s\wedge S_t}\|^2|\mathcal{F}_t] 
    &\le \mathbb{E}[(L\alpha_t+\|\Delta X_{S_t}\|)^21_{\{S_t<\infty\}}|\mathcal{F}_t]  \\
    &\le 2L^2 \alpha_t^2\mathbb{P}[S_t< \infty|\mathcal{F}_t]  +2k^2\alpha_t^2\times\kappa(\mathbb{P}[S_t< \infty|\mathcal{F}_t]).
  \end{align*}
  This implies (using that if $p\in [0,1]$, then $\kappa(p)\geq{p}$)
  that
  $\kappa(\mathbb{P}[S_t<\infty|\mathcal{F}_t]) \ge \frac{1}{2+8L^2k^2
    +8k^4}$. 
  This proves the lemma since $\kappa:[0,1]\to[0,1]$ is increasing.
\end{proof}

For $t\ge t_0$, set $U_t=\inf\{s>t:\, \|X_s\|\ge \rho\}$ (then,
$S_t\le U_t$ as soon as $\rho\ge L\alpha_{t_0}$) and denote by $H_t$ the
event
$\{\inf_{s\in [S_t,U_t]} \|X_s\|\ge
\frac{L\alpha_{S_t}}{4}\}\cap\{S_t<\infty\}$.

\begin{lemma}\label{lem:lem2}
  If $L\ge (2\varepsilon)\vee (2^3k\sqrt{2d})$ and $r\ge L\alpha_{t_0}$
  then, for all $t\ge t_0$, on the event $\{S_t<\infty\}$,
  $\mathbb{P}[H_t|\mathcal{F}_{S_t}]\ge 1/2$.
\end{lemma}
\begin{proof}
  Let us fix $t\ge t_0$ and set
  $T_t=\inf\{s>S_t:\, \|X_s\|<\frac{L\alpha_{S_t}}{4}\}$.  For
  $s\in [S_t,T_t\wedge U_t)$, It\^o's formula (Theorem 32 in
  \cite{Protter}) yields
  \begin{align*}
    \|X_s\|
    =&\; \|X_{S_t}\| +\int_{S_t}^s \left\langle \frac{X_{u-}}{\|X_{u-}\|},dX_u\right\rangle + \frac{1}{2} \int_{S_t}^s \sum_{i,j} \left(\frac{\delta_{i,j}}{\|X_{u-}\|} -\frac{X^i_ {u-}X^j_ {u-}}{\|X_{u-}\|^3}\right)d[X^i,X^j]_u \\
     &+ \sum_{{S_t}<u\le s} \left(\Delta \|X_u\| - \left\langle \frac{X_{u-}}{\|X_{u-}\|},\Delta X_u\right\rangle - \frac{1}{2} \sum_{i,j} \left(\frac{\delta_{i,j}}{\|X_{u-}\|} -\frac{X^i_ {u-}X^j_ {u-}}{\|X_{u-}\|^3}\right)\Delta X^i_u\Delta X^j_u\right). 
  \end{align*}
  One has that
  $\int_{S_t}^s \sum_{i,j} \left(\frac{\delta_{i,j}}{\|X_{u-}\|}
    -\frac{X^i_ {u-}X^j_ {u-}}{\|X_{u-}\|^3}\right)d[X^i,X^j]_u\ge
  \sum_{{S_t}<u\le s}\sum_{i,j} \left(\frac{\delta_{i,j}}{\|X_{u-}\|}
    -\frac{X^i_ {u-}X^j_ {u-}}{\|X_{u-}\|^3}\right)\Delta X^i_u\Delta
  X^j_u$.  One also has that
  $\Delta \|X_u\| - \left\langle \frac{X_{u-}}{\|X_{u-}\|},\Delta
    X_u\right\rangle\ge 0$ and that $\|X_{S_t}\|\ge L\alpha_{S_t}$. This
  implies that for $s\in [S_t,T_t\wedge U_t)$,
  \begin{align*}
    \|X_s\|
    \ge&\; L\alpha_{S_t} +\int_{S_t}^s \left\langle \frac{X_{u-}}{\|X_{u-}\|},dM_u\right\rangle + \int_{S_t}^s \left\langle \frac{X_{u-}}{\|X_{u-}\|},dR_u\right\rangle\\
    \ge&\; L\alpha_{S_t} +\int_{S_t}^s \left\langle \frac{X_{u-}}{\|X_{u-}\|},dM_u\right\rangle - (V_s-V_{S_t}).\\
    \ge&\; (L-\varepsilon)\alpha_{S_t} +\int_{S_t}^s \left\langle \frac{X_{u-}}{\|X_{u-}\|},dM_u\right\rangle.
  \end{align*}
  Let $(\widehat{M}_s^t,s\in [S_t,T_t\wedge U_t))$ be the martingale
  (stopped at $T_t\wedge U_t$) defined by
  $\widehat{M}_s^t=\int_{S_t}^s \left\langle
    \frac{X_{u-}}{\|X_{u-}\|},dM_u\right\rangle $ and set
  $I_t=\inf_{s\in [S_t,T_t\wedge U_t)}\widehat{M}_s^t$.  Then, for
  $s\in [S_t,T_t\wedge U_t)$ (using that
  $L- \varepsilon\ge \frac{L}{2}$)
  \begin{align*}
    \inf_{s\in [S_t,T_t\wedge U_t)}\|X_s\|
    \ge&\;   I_t + (L- \varepsilon)\alpha_{S_t} \ge I_t + \frac{L\alpha_{S_t}}{2}.
  \end{align*}
  On the event $\{I_t\ge -\frac{L\alpha_{S_t}}{4}\}\cap\{S_t<\infty\}$,
  we have
  $\inf_{s\in [S_t,T_t\wedge U_t)}\|X_s\|\ge \frac{L\alpha_{S_t}}{4}$,
  $T_t\wedge U_t=U_t$ and the event $H_t$ is realized.  Therefore, on
  the event $\{S_t<\infty\}$,
  $\mathbb{P}[H_t|\mathcal{F}_{S_t}] \ge \mathbb{P}\left[\left. I_t\ge
      -\frac{L\alpha_{S_t}}{4}\right|\mathcal{F}_{S_t}\right]$.  By
  Doob's inequality, on the event $\{S_t<\infty\}$,
  \begin{align*}
    \mathbb{P}\left[\left. I_t< -\frac{L\alpha_{S_t}}{4}\right|\mathcal{F}_{S_t}\right]
    &\le 4 \times \frac{4^2}{L^2\alpha^2_{S_t}}\mathbb{E}[[\widehat{M}^t]_{T_t\wedge U_t}|\mathcal{F}_{S_t}]\\
    &\le \frac{2^6d}{L^2\alpha^2_{S_t}}\mathbb{E}[[M]_{T_t\wedge U_t}-[M]_{S_t}|\mathcal{F}_{S_t}]\\
    &\le \frac{2^6d}{L^2\alpha^2_{S_t}}\mathbb{E}[\langle M\rangle_{T_t\wedge U_t}-\langle M\rangle_{S_t}|\mathcal{F}_{S_t}]\\
    &\le \frac{2^6dk^2}{L^2} \le \frac{1}{2}
  \end{align*}
  using that { $L\ge 2^3k\sqrt{2d}$}. This proves the lemma.
\end{proof}

\begin{proof}[Proof of Proposition \ref{prop:nonCV}] 
  Choose $L=2^3k\sqrt{2d}$. Then one has that $L\ge 2\varepsilon$ and that
  $(2L+\varepsilon)\varepsilon\le 2^2 L \varepsilon \le (2k)^{-2}$.  We
  also have that $\rho\ge L\alpha_{t_0}$.  The two previous lemmas can
  therefore be applied.  For $t\ge t_0$,
  \begin{align*}
    \mathbb{P}[H_t|\mathcal{F}_t]
    =& \mathbb{E}[1_{\{S_t<\infty\}}\mathbb{P}(H_t|\mathcal{F}_{S_t})|\mathcal{F}_t]\\
    \ge & \frac12 \mathbb{P}[S_t<\infty|\mathcal{F}_t]\ge \frac{p}{2}.
  \end{align*}

  Set $H=\{\limsup \|X_t\|>0\}$. We have for $s>t\ge t_0$,
  \begin{align*}
    \mathbb{P}[H|\mathcal{F}_t]
    \ge &\; \mathbb{P}[H\cap\{U_s=\infty\}|\mathcal{F}_t]\\
    \ge &\; \mathbb{P}[H_s\cap\{U_s=\infty\}|\mathcal{F}_t]\\
    \ge &\; \mathbb{P}[H_s|\mathcal{F}_t] - \mathbb{P}[U_s<\infty|\mathcal{F}_t].
  \end{align*}
  We have
  $\lim_{s\to\infty} \mathbb{P}[U_s<\infty|\mathcal{F}_t]=
  \mathbb{P}[A^c|\mathcal{F}_t]$, where $A$ is the event
  $\{\exists t; U_t=\infty\}$. Since,
  $\mathbb{P}[H_s|\mathcal{F}_t]= \mathbb{E}[\mathbb{P}[H_s|\mathcal{F}_s]|\mathcal{F}_t]\ge \frac{p}{2}$, we
  have
  \begin{align*}
    \mathbb{P}[H|\mathcal{F}_t]
    &\ge\; \frac{p}{2} - \mathbb{P}[A^c|\mathcal{F}_t]
  \end{align*}

  Since, a.s., $\lim_{t\to\infty} \mathbb{P}[H|\mathcal{F}_t]=1_{H}$ and
  $\lim_{t\to\infty} \mathbb{P}[A^c|\mathcal{F}_t]=1_{A^c}$,
  \begin{equation*}
    1_{H} \geq  \frac{p}{2}-1_{A^c} \qquad \hbox{ a.s.}
  \end{equation*}
  This implies that a.s.,
  $A\subset H$.  But since $H^c\subset A$, we have that a.s.
  $H^c\subset H$, which is possible only if $\mathbb{P}(H)=1$.
\end{proof}

\subsubsection{Proof of Theorem \ref{thm:th2}}
Proposition \ref{prop:nonCV} applied to the process $X'$ introduced at
the end of subsection \ref{sec:Simplification} and for the event
$\Gamma_{t_0,k,\varepsilon}$ shows that
$\mathbb{P}\left[\Gamma_{t_0,k,\varepsilon}\cap \{\lim_{t\to\infty}
  X_t=0\}\right]=0$ if $\alpha_{t_0}\le \frac{\rho}{2^3k\sqrt{2d}}$ and if
$\varepsilon\le \frac{1}{2^7k^3\sqrt{2d}}$. As noticed in subsection
\ref{sec:Simplification}, this suffices to prove Theorem \ref{thm:th2}.

\subsubsection{Proof of \textcolor{black}{Proposition~\ref{thm:th22}}}
The framework is the same as for Theorem \ref{thm:th2}, but in
\textcolor{black}{Proposition~\ref{thm:th22}}, we have
$R_{t}=R'_{t}+\int_{0}^{t}r''_{s}\,\mathrm{d}s$ and conditions \eqref{hyp:F1} and \eqref{hyp:calpha} are replaced by
\eqref{hyp:F12}, \eqref{hyp:calpha2} and \eqref{hyp:calpha3}.

\begin{proof}[Proof of \textcolor{black}{Proposition~\ref{thm:th22}}]
  It suffices to follow the lines of the proof of Theorem
  \ref{thm:th2}. The only changes are essentially the following ones:
  \begin{itemize}
  \item In the proof of Lemma \ref{lem:lem1}, equation \eqref{eq:intXX}
    has to be replaced by
    \begin{align}
      \nonumber\int_t^{s\wedge S_t} \langle X_{u-},dX_u\rangle
      &\;  = \int_t^{s\wedge S_t} \langle X_{u-},F_u\rangle du + \int_t^{s\wedge S_t} \langle X_{u-},dM_u\rangle
        + \int_t^{s\wedge S_t} \langle X_{u-},dR_u\rangle\\
      &\;\ge \int_t^{s\wedge S_t} \langle X_{u-},dM_u\rangle
        + \int_t^{s\wedge S_t} \langle X_{u-},dR'_u\rangle
        + \int_t^{s\wedge S_t} \gamma_u \|X_{u-}\|\big(\|X_{u-}\| -k\gamma_u^{\nu}\big) du.
    \end{align}
    and (separating the cases $\|X_{u-}\|\ge k\gamma_u^{\nu}$ and
    $\|X_{u-}\|<k\gamma_u^{\nu}$) we underestimate the last integral
    by
    $-k^2\int_t^{s\wedge S_t} \gamma_u^{1+2\nu}\mathrm{d}u=o(\alpha_t^2)$ since \(\nu>\frac{1}{2}\).
  \item In the proof of Lemma \ref{lem:lem2}, when one underestimates
    $\|X_s\|$ for $s\in [S_t,T_t\wedge U_t]$ we must add the term
    $\int_{S_t}^s\left\langle\frac{X_{u-}}{\|X_{u-}\|},F_u+r''_u\right\rangle\,\mathrm{d}u$
    which can be underestimated by (using Cauchy-Schwartz inequality for
    the second term)
    \begin{eqnarray*}
      \int_{S_t}^s \gamma_u\big(\|X_{u-}\| -k\gamma^\nu_u\big) du
      &\ge& \frac{L\alpha_{S_t}}{4} \int_{S_t}^s \gamma_u du - k \sqrt{\int_{S_t}^s \gamma_u^{2} du} \sqrt{\int_{S_t}^s \gamma_u^{2\nu} du}\\
      &\ge& \frac{L\alpha_{S_t}}{4} \int_{S_t}^s \gamma_u du - k^2 \alpha_{S_t} \sqrt{\int_{S_t}^s \gamma_u^{2\nu} du}\\
    \end{eqnarray*}
    and we use the fact that (since $\nu>1/2$)
    $\sqrt{\int_{S_t}^s \gamma_u^{2\nu} du}=(1+\int_{S_t}^s \gamma_u
    du\big) \times o(1)$.
  \end{itemize}
\end{proof}

\subsection{Proof of \textcolor{black}{Corollary~\ref{thm:th3b}}}
\label{sec:nonCVrepulsive}
In \textcolor{black}{Corollary~\ref{thm:th3b}}, $H$ is a repulsive matrix, i.e. a matrix such
that its eigenvalues all have a positive real part. We also have that
for all $t\ge 0$, $\gamma_t\ge 0$ and $H_t$ is a
$\mathcal{M}_d(\mathbb{R})$-valued random variable. And it is supposed
that
\begin{enumerate}[\itshape{}(i)]
\item a.s. on $\Gamma$, $\lim_{t\to\infty}H_t=H$,
\item there is a finite random variable $T_0$ such that, a.s. on
  $\Gamma$, for all $t\ge T_0$, $F_t=\gamma_t H_tX_{t-}$
\end{enumerate}
and that conditions \eqref{hyp:<M>1}, \eqref{hyp:<M>2},
\eqref{hyp:calpha} and \eqref{hyp:jumps} are satisfied on $\Gamma$.

The proof of \textcolor{black}{Corollary~\ref{thm:th3b}} relies on
Theorem \ref{thm:th2} and on the following lemma:
\begin{lemma}\label{lem:repulsive}
  There are $\lambda>0$ and an inner product $\langle\cdot,\cdot\rangle$
  such that $\langle Hx,x\rangle\ge \lambda\|x\|^2$, with
  $\|x\|^2=\langle x,x\rangle$.
\end{lemma}
Before proving this lemma, let us prove \textcolor{black}{Corollary~\ref{thm:th3b}}:
\begin{proof}[Proof of \textcolor{black}{Corollary~\ref{thm:th3b}}]
  Lemma \ref{lem:repulsive} ensures the existence of $\lambda>0$ and of
  an inner product $\langle\cdot,\cdot\rangle$ such that
  $\langle Hx,x\rangle\ge \lambda\|x\|^2$.  Since a.s. on
  $\Gamma\cap\{\lim_{t\to\infty}X_t=0\}$, $\lim_{t\to\infty}H_t=H$, we
  have that a.s. on $\Gamma\cap\{\lim_{t\to\infty}X_t=0\}$, there is
  $t_0\ge T_0$ such that for all $t\ge t_0$,
  $\|H_t-H\|\le \frac{\lambda}{2}$ (where for a matrix $A$,
  $\|A\|:=\sup_{x;\|x\|=1} \|Ax\|$). Then, a.s. on
  $\Gamma\cap\{\lim_{t\to\infty}X_t=0\}$, it holds that for all
  $t\ge t_0$
  \begin{align*}
    \langle F_t,X_{t-}\rangle=&\;\gamma_t\langle H_t X_{t-},X_{t-}\rangle\\
    =&\;\gamma_t\langle H X_{t-},X_{t-}\rangle + \gamma_t\langle (H_t - H) X_{t-},X_{t-}\rangle
    \\
    \ge&\;  \gamma_t\big(\lambda - \lambda/2\big) \|X_{t-}\|^2 \ge 0.
  \end{align*} 
  Then one concludes applying Theorem \ref{thm:th2}.
\end{proof}

\begin{proof}[Proof of Lemma \ref{lem:repulsive}]
  Let $\lambda\in (0,\infty)$ be such that the real part of every
  eigenvalue of $H$ is greater than $2\lambda$.  The real part of a
  complex number $z$ will be denoted $\mathcal{R}(z)$.

  Using the Jordan reduction of $H$, for all
  $\varepsilon\in (0,\lambda]$, there are
  $P\in \mathcal{M}_d(\mathbb{C})$ an invertible matrix and
  $T\in \mathcal{M}_d(\mathbb{C})$ a triangular matrix such that
  $H=PTP^{-1}$ and such that $T_{i,j}=0$ if $j\not\in\{i,i+1\}$,
  $\lambda_i:=T_{i,i}$ is an eigenvalue of $H$ and
  $\varepsilon_i:= T_{i,i+1}\in\{0,\varepsilon\}$.  For $1\le i\le d$,
  let $e_i$ be $i$-th column of $P$. Then $\mathcal{B}:=(e_1,\dots,e_d)$
  is a basis of $\mathbb{C}^d$. Define the sesquilinear form on
  $\mathbb{C}^d$ defined by
  $b(x,x')=\bar{X}^T X'=\sum_{i=1}^d \bar{X}_i X'_i$, where $X$ and $X'$
  are the vectors of the coordinates of $x$ and $x'$ in $\mathcal{B}$.
  Set $Q=P^{-1}$ and define the inner product
  $\langle\cdot,\cdot\rangle$ on $\mathbb{R}^d$ by
  $\langle x,x'\rangle = \mathcal{R}\big( b(Qx,Qx')\big)$.  Denote by
  $\|\cdot\|$ the norm associated to this inner product.  Then, for
  $(x,x')\in\mathbb{R}^d\times \mathbb{R}^d$,
  \begin{align*}
    \langle x,x'\rangle 
    =&\; \mathcal{R}\big( x^T \bar{Q}^T Qx'\big)
    = x^T S x'
  \end{align*}
  where $S=\frac{1}{2}\big( \bar{Q}^T Q + Q^T\bar{Q}\big)$ is a
  symmetric matrix in $\mathcal{M}_d(\mathbb{R})$.
  For $x\in \mathbb{R}^d$ and $z:=Qx$,
  \begin{align*}
    \langle x,Hx\rangle 
    = &\; \mathcal{R}\big( x^T\bar{Q}^T Q H x\big)=  \mathcal{R}\big( x^T\bar{Q}^T TQ x\big)\\
    = &\; \mathcal{R}\big( \bar{z}^T Tz\big) =  \mathcal{R}\big(\sum_{i=1}^d
        \lambda_i|z_i|^2 + \sum_{i=1}^d \varepsilon_i  \bar{z}_i z_{i+1}\big)\\
    \ge &\; 2\lambda \sum_i |z_i|^2 - {\varepsilon}\sum_i |z_i|\times |z_{i+1}| \\
    \ge &\; 2\lambda \sum_i |z_i|^2 - \frac{\varepsilon}{2}\sum_i \big(|z_i|^2+ |z_{i+1}|^2)  \\
    \ge &\; (2\lambda- \varepsilon) \sum_i |z_i|^2 \ge\lambda\|x\|^2.
  \end{align*}
\end{proof}

\subsection{Proof of Theorem \ref{thm:th4}}
\label{sec:nonCVunstable}
In Theorem \ref{thm:th4}, we let $x^*\in \mathbb{R}^d$ and
$f:\mathbb{R}^d\to\mathbb{R}^d$ be a function such that $f$ is $C^{1}$
in a convex neighborhood $\mathcal{N}^*$ of $x^*$, and we suppose that
$x^*$ is an equilibrium for $f$, i.e. $f(x^*)=0$.  Let us recall the
construction given just before Theorem~\ref{thm:th4}: there are integers
$(\delta^+,\delta^-)\in \{0,\dots,d\}^2$ with $\delta^++\delta^-=d$,
an invertible matrix $P\in\mathcal{M}_d(\mathbb{R})$, a matrix
$H^+\in\mathcal{M}_{\delta^+}(\mathbb{R})$ and a matrix
$H^-\in\mathcal{M}_{\delta^-}(\mathbb{R})$ such that
\begin{itemize}
\item
  $P^{-1} Df(x^*) P = \hbox{diag}[H^+,H^-] = \begin{pmatrix} H^+ & 0\\
    0&H^-
  \end{pmatrix}$;
\item The eigenvalues of $H^+$ all have a positive real part, i.e. $H^+$
  is repulsive;
\item The eigenvalues of $H^-$ all have a non-positive real part,
  i.e. $H^-$ is non-repulsive.
\end{itemize}
Note that if $x^*$ is unstable then $\delta^+\ge 1$ and if $x^*$ is
repulsive then $\delta^+=d$.  For $x\in\mathbb{R}^d$, set
$y:=P^{-1}x$. Then there are \(y^{+}\in\mathbb{R}^{\delta^{+}}\) and
\(y^{-}\in\mathbb{R}^{\delta^{-}}\) such that
$y=\begin{pmatrix} y^+\\y^-\end{pmatrix}$ and
$P^{-1} Df(x^*) Py=\begin{pmatrix} H^+y^+\\H^-y^-\end{pmatrix}$.  Set
$M^\pm=(P^{-1}M)^\pm$ and $R^\pm=(P^{-1}R)^\pm$.

\smallskip We suppose that $F_t=\gamma_t f(X_t)$ with $\gamma_t\ge 0$
and that a.s. on $\Gamma$, conditions \eqref{hyp:<M>1} and
\eqref{hyp:<M>2} are satisfied by $M^+$, and that \eqref{hyp:calpha} and
\eqref{hyp:jumps} are satisfied.  We also suppose that a.s. on $\Gamma$,
$\langle M^-\rangle_{t,\tau(t)}=O(\langle M^+\rangle_{t,\tau(t)})$.

Theorem \ref{thm:th4} states that if the hypotheses described above in
this section are satisfied and if one of the three following conditions
is satisfied
\begin{enumerate}[\itshape (i)]
\item $x^*$ is repulsive;
\item $x^*$ is unstable, $f\in C^{1+\nu}(\mathcal{N}^*)$ for some
  $\nu\in (0,1]$, $M$ is a purely discontinuous martingale and
  $\sum_{s>t}\|\Delta M^-_s\|^{1+\nu}=o(\alpha_t)$ on
  $\Gamma\cap\{\lim_{t\to\infty} X_t=x^*\}$.
\item $x^*$ is unstable, $f\in C^2(\mathcal{N}^*)$ and
  $\sum_{s>t}\|\Delta M^-_s\|^{2}=o(\alpha_t)$ on
  $\Gamma\cap\{\lim_{t\to\infty} X_t=x^*\}$;
\end{enumerate}
Then it holds that $\mathbb{P}[\Gamma\cap\{\lim_{t\to\infty}X_t=x^*\}]=0$.

\begin{proof}[Proof of Theorem \ref{thm:th4}]
  To simplify the notation, we suppose that $x^*=0$.

  Suppose first that \textit{(i)} is satisfied, i.e. $0$ is
  repulsive. Then \(\delta^{+}=d\) and \(M=M^{+}\) (and \(M^{-}=0\)).
  Set $H=Df(0)$. Then one has
  $$f(x)=f(0)+\left[ Df(0)+\int_0^1 [Df(tx)-Df(0)] dt\right]\cdot
  x=[H+r(x)]x,$$ with $r(x)=\int_0^1 [Df(tx)-Df(0)] dt=o(1)$.  For
  $t\ge 0$, set $H_t=H+r(X_{t-})$.  Then on
  $\Gamma\cap\{\lim_{t\to\infty} X_t=0\}$, $\lim_{t\to\infty}H_t=H$ and
  one can conclude applying \textcolor{black}{Corollary~\ref{thm:th3b}}.

  Suppose now that $0$ is unstable and that condition \textit{(ii)} or
  \textit{(iii)} is satisfied. We first rectify the vector field
  \(f(x)\) and define new coordinates \((u^{+},u^{-})\) such that the
  local stable manifold is an open subset of \(\{u^{+}=0\}\) and such
  that \(u^{-}\) is a linear transformation of \(x\). For this
  purpose, we follow Section I.4 in \cite{Brandiere96} and chapter IX
  section 5 in \cite{Hartman}.  The equilibrium $0$ being unstable,
  $\delta^+\ge 1$.  For $y\in \mathbb{R}^d$, write
  $y=\begin{pmatrix} y^+\\ y^-
  \end{pmatrix}$ with $y^+\in\mathbb{R}^{\delta^+}$ and
  $y^-\in\mathbb{R}^{\delta^-}$.  Set
  $h(y):=P^{-1}f(Py)=\begin{pmatrix} h^+(y)\\h^-(y)
  \end{pmatrix}$.  The function $h$ is $C^{1+\nu}$ in a neighborhood of
  $0$, $h(0)=0$ and $Dh(0)=\hbox{diag}[H^+,H^-]$.  We then have that
  $Y_t:=P^{-1}X_t$ satisfies
  \begin{equation}
    \left\{
      \begin{array}{l}
        Y^+_t = Y^+_0 +\int_0^t \gamma_s h^+(Y_{s-}) ds + (P^{-1}M_t)^+ + (P^{-1}R_t)^+\\ 
        Y^-_t = Y^+_0 +\int_0^t \gamma_s h^-(Y_{s-}) ds + (P^{-1}M_t)^- + (P^{-1}R_t)^-
      \end{array}
    \right.
  \end{equation}

  Following \cite{Hartman} (Lemma 5.1, Example 5.1 and Corollary 5.2),
  there are $\mathcal{N^+}$ and $\mathcal{N^-}$, convex neighborhoods of
  $0$ respectively in $\mathbb{R}^{\delta^+}$ and
  $\mathbb{R}^{\delta^-}$, and
  $g:\mathbb{R}^{\delta^-}\to \mathbb{R}^{\delta^+}$ a $C^1$-function
  such that $g\in C^{1+ \nu}(\mathcal{N}^-)$,
  \begin{itemize}
  \item $g(0)=0$ and
    $Dg(0)=0$;
  \item If $y\in \mathcal{N}:=\mathcal{N}^+\times \mathcal{N}^-$ is such
    that $y^+=g(y^-)$, then $h^+(y)=Dg(y^-)h^-(y)$.
  \end{itemize}
  The last assumption ensures that if $y_t$ satisfies
  $\frac{dy_t}{dt}=h(y_t)$ with $y(0)\in\mathcal{N}$ such that
  $u^+_0:=y^+_0-g(y^-_0)=0$, then as long as $y_t\in \mathcal{N}$, one
  has $u^+_t:=y^+_t-g(y^-_t)=0$.  Set
  $\mathcal{N}^*:=\{x:P^{-1}x\in\mathcal{N}\}$.  Then the set $K$ of all
  $x\in \mathcal{N}^*$ such that $y^+=g(y^-)$, where $y=P^{-1}x$, is the
  (local) center-stable manifold of $f$ at $0$.

  For $(u^+,u^-) \in \mathbb{R}^{\delta^+}\times \mathbb{R}^{\delta^-}$,
  let $y\in \mathbb{R}^d$ be such that $y^-=u^-$, $y^+=u^++g(u^-)$ and
  let
  $F^+:\mathbb{R}^\delta\times \mathbb{R}^{d-\delta}\to
  \mathbb{R}^\delta$ be defined by $F^+(u^+,u^-)=h^+(y)-Dg(y^-)h^-(y)$.
  If $y\in \mathcal{N}$ then $F^+(0,u^-)=0$ and, denoting by $D^+$ the
  differential with respect to $u^+$,
  \begin{align*}
    F^+(u^+,u^-)=\left[\int_0^1 D^+F^+(tu^+,u^-) dt\right] u^+= [H^++\Delta(u^+,u^-)] u^+,
  \end{align*}
  where $\Delta(u^+,u^-)$ is a matrix satisfying
  $\|\Delta(u^+,u^-)\| = O(\|y\|^ \nu)$ as $y\to 0$, with $ \nu=1$ when
  condition \textit{(iii)} is satisfied.

  Set now $U^+_t=Y^+_t-g(Y^-_t)$, $U^-_t=Y^-_t$ and
  $H^+_t=H^++\Delta(U^+_{t-},U^-_{t-})$. Then on
  $\Gamma\cap \{\lim_{t\to\infty} X_t=0\}$, we have that
  $\lim_{t\to\infty} H^+_t=H^+$.

  In the case \textit{(ii)} is satisfied, $M$ is a pure jump
  martingale and \(U^{+}\) satisfies (recall that
  $M^\pm=(P^{-1}M)^\pm$ and $R^\pm=(P^{-1}R)^\pm$)
  \begin{align*}
    U^+_t =&\; U^+_0 +\int_0^t F^u_s ds + M^u_t + R^u_t 
  \end{align*} 
  where
  \begin{align*}
    &F^u_t=\gamma_t \left[h^+(Y_{t-}) - Dg(Y^-_{t-}) h^-(Y_{t-})  \right]\\
    &M^u_t=
      \left[ M_t^+ - \int_0^t Dg(Y^-_{s-})dM^-_s\right]\\
    &R^u_t= \left[ R_t^+- \int_0^t Dg(Y^-_{s-})dR^-_s\right] - \sum_{0<s\le t} \big(\Delta g(Y^-_s) - Dg(Y^-_{s-})\Delta Y^-_{s}\big).
  \end{align*}

  It is straightforward to check that there is a finite random variable
  $T$ such that on $\Gamma\cap \{\lim_{t\to\infty} X_t=0\}$, we have
  that $Y_t\in\mathcal{N}$ for all $t\ge T$.  Since on
  $\Gamma\cap \{\lim_{t\to\infty} X_t=0\}$,
  $F^u_t=\gamma_t H^+_t U^+_{t-}$ for all $t\ge T$,
  $\lim_{t\to\infty} H^+_t=H^+$ and since $H^+$ is a repulsive matrix,
  in order to apply \textcolor{black}{Corollary~\ref{thm:th3b}} to $U^+$ and $\Gamma$, it
  suffices to check that conditions \eqref{hyp:<M>1}, \eqref{hyp:<M>2},
  \eqref{hyp:calpha} and \eqref{hyp:jumps} are satisfied by $U^+$, $M^u$
  and $R^u$.

  We have that on $\Gamma\cap\{\lim_{t\to\infty} X_t=0\}$ (using that
  $Dg(0)=0$ and that $Dg$ is continuous),
  $$\left\langle \int_0^\cdot
    Dg(Y_{s-}^-)dM_s^-\right\rangle_{t,\tau(t)}=o\big(\langle
  M^+\rangle_{t,\tau(t)}\big) \quad\text{ and }\quad \left\langle
    \int_0^\cdot
    Dg(Y_{s-}^-)dM_s^-\right\rangle_{t,\infty}=o\big(\langle
  M^+\rangle_{t,\infty}\big).$$ Therefore on
  $\Gamma\cap\{\lim_{t\to\infty} X_t=0\}$ , we have first
  $\langle M^u\rangle_{t,\tau(t)}\ge \frac{1}{2}\langle
  M^+\rangle_{t,\tau(t)} +o(\langle M^+\rangle_{t,\tau(t)})$ and, since
  $\alpha^2_t-\alpha^2_{\tau(t)}=O(\langle M^+\rangle_{t,\tau(t)})$,
  condition \eqref{hyp:<M>1} is satisfied by $M^u$.  Secondly, on
  $\Gamma\cap\{\lim_{t\to\infty} X_t=0\}$ we have
  $\langle M^u\rangle_{t,\infty}\le 2\langle M^+\rangle_{t,\infty}
  +o(\langle M^+\rangle_{t,\infty})=0(\alpha^2_t)$ and condition
  \eqref{hyp:<M>2} is satisfied by $M^u$.

  Set $V^u_t=V(R^{u},(0,t])$ the variation of $R^u$ on $[0,t]$.  On
  $\Gamma\cap\{\lim_{t\to\infty} X_t=0\}$,
  $\big\|\Delta g(Y^-_s) - Dg(Y^-_{s-})\Delta Y^-_{s}\big\|=O\big(
  \|\Delta Y^-_s\|^{1+ \nu}\big)=O\big( \|\Delta M^-_s\|^{1+ \nu} +
  \|\Delta R^-_s\|^{1+ \nu}\big)$ and a.s. on
  $\Gamma\cap\{\lim_{t\to\infty} X_t=0\}$,
  \[
    V^u_{t,\infty}= O\left(\sum_{s>t}\|\Delta M^-_s\|^{1+ \nu} +
      V_{t,\infty} \right) = o(\alpha_t),
  \]
  i.e. condition \eqref{hyp:calpha} is satisfied. We have
  $\|\Delta U^+_t\|\le C \|\Delta X_t\|$ for some non-random constant
  $C$ and so, since $X$ satisfies \eqref{hyp:jumps}, $U^+$ also
  satisfies \eqref{hyp:jumps}.  And we conclude applying \textcolor{black}{Corollary~\ref{thm:th3b}}.

  In the case condition \textit{(iii)} is satisfied, the arguments are
  identical and will be omitted. The most essential change is that one
  can apply It\^o's formula since $g$ is $C^2$.
\end{proof}

\begin{remark}
  \begin{itemize}
  \item If in the previous construction $g=0$, i.e. when the local
    center-stable manifold of $f$ at $x^*$ is a vector space, then the
    conditions on $M^-$ can be
    removed
  \item If $M$ is a continuous martingale, then the condition that
    $\sum_{s>t}\|\Delta M^-_s\|^{2}=o(\alpha_t)$ a.s. on
    $\Gamma\cap\{\lim X_t=x^*\}$ is always satisfied and can be removed
    in \textit{(iii)}.
  \end{itemize}
\end{remark}

\subsection{Proof of Theorem \ref{thm:th5}}
\label{sec:lastnonCVthm}

In Theorem \ref{thm:th5}, we suppose that a.s. on $\Gamma$, $M$ and $R$
satisfy \eqref{hyp:<M>2} and \eqref{hyp:calpha}.

\subsubsection{Assumptions on $F_t$ and the unstable manifold.}
\label{sec:Ft-assumption}

We suppose that $F_t=\gamma_tf(X_{t-})$, where
$\gamma:[0,\infty)\to[0,\infty)$ is a measurable function and
$f:\mathbb{R}^d\to\mathbb{R}^d$ is a measurable vector field, which is
$C^{1}$ in a neighborhood of an unstable equilibrium $x^*$ for $\varphi$
the flow generated by $f$.  Without loss of generality, we suppose that
$x^*=0$.  Also, we will suppose that the equilibrium $x^*$ is
hyperbolic, i.e. the eigenvalues of $Df(x^*)$ all have a non-zero real
part (note that this assumption is not necessary in this subsection and
is only used in subsection \ref{subsection:th5}).

The rectification of \(f\) given below is similar to the one given in
Section~\ref{sec:nonCVunstable}, but here the new coordinates
\((u^{+},u^{-})\) are such that the local unstable manifold is an open
subset of \(\{u^{-}=0\}\) and such that \(u^{+}\) is a linear
transformation of \(x\), and the function \(g\) maps
\(\mathcal{N}^{+}\) onto \(\mathcal{N}^{-}\) (whereas in
Section~\ref{sec:nonCVunstable}, \(g\) maps \(\mathcal{N}^{-}\) onto
\(\mathcal{N}^{+}\)). Following \cite{Hartman}, there are integers
$\delta^+\in \{1,\dots,d\}$, $\delta^-=d-\delta^+$, an invertible
matrix $P\in\mathcal{M}_d(\mathbb{R})$, matrices
$H^+\in\mathcal{M}_{\delta^+}(\mathbb{R})$ and
$H^-\in\mathcal{M}_{\delta^-}(\mathbb{R})$ such that
\begin{itemize}
\item $P^{-1} Df(x^*) P = \hbox{diag}[H^+,H^-]$;
\item The eigenvalues of $H^+$ all have a positive real part and the
  eigenvalues of $H^-$ all have a negative real part, i.e. $H^+$ is
  repulsive and $H^-$ is attractive.
\end{itemize}

Let $\mu$ be the smallest real number such that every eigenvalue of
$H^-$ has a real part less or equal to $\mu$. Then, $x^*$ being
hyperbolic, $\mu<0$.

For $x\in\mathbb{R}^d$, set $y:=P^{-1}x$. There is
$(y^+,y^-)\in\mathbb{R}^{\delta^+}\times\mathbb{R}^{\delta^-}$ such that
$y=\begin{pmatrix} y^+\\y^-\end{pmatrix}$ and
$P^{-1} Df(x^*) Py=\begin{pmatrix} H^+y^+\\H^-y^-\end{pmatrix}$.  Set
$M^\pm=(P^{-1}M)^\pm$ and $R^\pm=(P^{-1}R)^\pm$.  Set also
$h(y):=P^{-1}f(Py)=\begin{pmatrix} h^+(y)\\h^-(y)
\end{pmatrix}$.  Then $h$ is $C^{1}$ in a neighborhood of $0$, $h(0)=0$
and $Dh(0)=\hbox{diag}[H^+,H^-]$.

Following \cite{Hartman} (Lemma 5.1, Example 5.1 and Corollary 5.2),
there are $\mathcal{N^+}$ and $\mathcal{N^-}$, convex neighborhoods of
$0$ respectively in $\mathbb{R}^{\delta^+}$ and $\mathbb{R}^{\delta^-}$
and $g:\mathcal{N}^+\to \mathcal{N}^-$ a $C^1$ function such that
\begin{itemize}
\item $g(0)=0$ and
  $Dg(0)=0$;
\item If $y\in \mathcal{N}:=\mathcal{N}^+\times \mathcal{N}^-$ is such
  that $y^-=g(y^+)$, then $h^-(y)=Dg(y^+)h^+(y)$.
\end{itemize}
The last assumption ensures that if $y_t$ satisfies
$\frac{dy_t}{dt}=h(y_t)$ with $y(0)\in\mathcal{N}$ such that
$u^-_0:=y^-_0-g(y^+_0)=0$, then as long as $y_t\in \mathcal{N}$, one has
$u^-_t:=y^-_t-g(y^+_t)=0$.  Set
$\mathcal{N}^*:=\{x:P^{-1}x\in\mathcal{N}\}$.  Then the set $K$ of all
$x\in \mathcal{N}^*$ such that $y^-=g(y^+)$, where $y=P^{-1}x$, is the
(local) unstable manifold of $f$ at $0$.  For $x\in\mathcal{N}^*$, set
$u=\begin{pmatrix} u^+\\u^-\end{pmatrix}=\begin{pmatrix}
  y^+\\y^--g(y^+)\end{pmatrix}$, where set
$y=\begin{pmatrix}
  y^+\\y^-\end{pmatrix}=P^{-1}x$. 
Then $\psi:\mathcal{N}^*\to \psi(\mathcal{N}^*)$ defined by $\psi(x)=u$
is a $C^1$-diffeomorphism.  We denote by $d(x,x')$ the distance in
$\mathcal{N}^*$ defined\footnote{For $u\in \mathbb{R}^\delta$,
  $\delta\ge 1$, $\|u\|^2=u^T u$.} by $d(x,x')=\|\psi(x)-\psi(x')\|$.
Note that $d(x,K)=\|u^-\|$ (since $K$ is the set of all
$x\in \mathcal{N}^*$ such that $u^-=0$).  Let also $k$ be the vector
field defined on $\psi(\mathcal{N}^*)$ by
\[
  \begin{cases}
    k^+(u^+,u^-)=h^+(u^+,u^-+g(u^+)),\\
    k^-(u^+,u^-)=h^-(u^+,u^-+g(u^+))-Dg(u^+)h^+(u^+,u^-+g(u^+)).
  \end{cases}
\]
Then $k^+$ and $k^-$ are $C^1$ and $k^-(u^+,0)=0$.

\begin{lemma}\label{lem:muattract}
  For all $\varepsilon\in (0,-\mu)$, there is a neighborhood
  $\mathcal{N}_\varepsilon\subset \mathcal{N}^*$ of $0$ such that $K$
  attracts $\mathcal{N}_\varepsilon$ at rate $\mu+\varepsilon$ in the sense that
  \begin{equation}
    d(\varphi_t(x),K) \le e^{(\mu+\varepsilon) t} d(x,K).
  \end{equation}
  for all $x\in \mathcal{N}_\varepsilon$ and $t\in [0,t_\varepsilon(x)]$ where
  $t_\varepsilon(x)=\inf\{t>0:\, \varphi_t(x)\notin \mathcal{N}_\varepsilon\}$.
\end{lemma}
\begin{proof}
  For $x\in\mathcal{N}^*$ and
  $0\le t\le \inf\{t>0:\, \varphi_t(x)\notin \mathcal{N}^*\}$, set
  $x_t=\varphi_t(x)$ and define $u_t=\psi(x_t)$ as above. Then
  \[
    \begin{cases}
      \frac{du^+_t}{dt} = k^+(u^+_t,u^-_t),\\
      \frac{du^-_t}{dt} = k^-(u^+_t,u^-_t).
    \end{cases}
  \]
  We have (denoting by $D^-$ the differential with respect to $u^-$)
  \begin{eqnarray*}
    k^-(u^+,u^-) &=& k^-(u^+,0)+D^-k^-(u^+,0) u^-
                     + o(\|u^-\|)\\
                 &=& k^-(u^+,0)+D^-k^-(0,0)u^- + o(\|u^-\|)\\
                 &=& H^-u^- + o(\|u^-\|).
  \end{eqnarray*}
  We thus have that
  \[
    (u^-)^T\, k^-(u^+,u^-)=(u^-)^T\, H^-u^-+o(\|u^-\|^2).
  \]
  It holds that $(u^-)^T\, H^-u^-\le \mu\|u^-\|^2$ and that there is a
  neighborhood $\mathcal{N}_\varepsilon$ sufficiently small such that for
  all $x\in \mathcal{N}_\varepsilon$,
  $(u^-)^T\, k^-(u^+,u^-)\le (\mu+\varepsilon)\|u^-\|^2$.

  We thus have that for all $t\in [0,t_\varepsilon(x)]$,
  \begin{eqnarray*}
    \frac{d\|u^-_t\|^2}{dt}
    &=& 2 (u^-_t)^T\, k^-(u^+_t,u^-_t)\\
    &\le& 2 (\mu+\varepsilon) \|u^-_t\|^2
  \end{eqnarray*}
  and so that for all $t\in [0,t_\varepsilon(x)]$,
  $\|u^-_t\|^2\le e^{2(\mu+\varepsilon)t}\|u^-_0\|^2$. This proves this
  lemma.
\end{proof}

There is one last assumption that will have to be done on the vector
field $f$: there is $\nu>0$ and a constant $C_\nu<\infty$ such that for
all $x\in \mathcal{N}^*$, setting $u=\psi(x)$,
\begin{equation}\label{hyp:h+r}
  \|k^+(u^+,u^-)-k^+(u^+,0)\|\le C_\nu \|u^-\|^{1+\nu}.
\end{equation}
Recall that $\|u^-\|=d(x,K)$ and that, setting $y=P^{-1}x$,
$h^+(y^+,y^-)=k^+(u^+,u^-)$ and $u^-=y^--g(y^+)$.

In the following, set
$\mathfrak{h}:\mathcal{N}^+\to\mathbb{R}^{\delta^+}$ the vector field
defined by $\mathfrak{h}(u^+)=k^+(u^+,0)$. This vector field is $C^1$ in
$\mathcal{N}^+$ and $0$ is a repulsive equilibrium for $\mathfrak{h}$.

\begin{example}
  Let us given $\nu>0$, $\mu<0$, $\delta^+\ge 1$ and $\delta^-\ge 1$
  with $\delta^++\delta^-=d$.  For $x\in\mathbb{R}^d$, set
  $x^\nu\in \mathbb{R}^d$ such that $x^\nu_i=x_i$ if $i\le \delta^+$ and
  $x^\nu_i=|x_i|^{1+\nu}$ otherwise.
  Let $f$ be a $C^1$ vector field on $\mathbb{R}^d$ such that in a
  compact neighborhood $\mathcal{N}$ of $0$, we have
  \[
    f_i(x)=\begin{cases}
      F^+_i(x^\nu) & \hbox{ if } i\le \delta^+\\
      x_iF^-_i(x) & \hbox{ otherwise }\\
    \end{cases}
  \]
  where for $i\le \delta^+$, $F^+_i$ is $C^1$ with $F^+_i(0)=0$, and for
  $i\ge \delta^++1$, $F^-_i$ is a continuous function such that
  $F^-_i(0)=\mu$.  Then $0$ is an equilibrium for $f$ and
  $Df(0)=\hbox{diag}[H^+,H^-]$, with $H^+_{i,j}=\partial_j F_i^+(0)$ and
  $H^-_{ij}=\mu \delta_{ij}$. Since $\mu<0$, $H^-$ is attractive. If
  $F^+$ is such that $H^+$ is repulsive then we are in the framework
  described above, with $P=I$, $g=0$, $k=h=f$ and
  $K=\{x\in\mathcal{N}:x^-=0\}$ (where for $x\in\mathcal{N}$,
  $x^+\in\mathbb{R}^{\delta^+}$ is defined by $x^+_i=x_i$ for
  $i\le\delta^+$ and $x^-\in\mathbb{R}^{\delta^-}$ is defined by
  $x^-_i=x_{\delta^++i}$ for $i\le\delta^-$).  Then, for all
  $\varepsilon\in (0,-\mu)$, $K$ attracts a neighborhood of $0$ at rate
  $\mu+\varepsilon$ (in the sense given in Lemma \ref{lem:muattract}), and
  since for all $i\le \delta^+$, $F^+_i$ is Lipschitz in $\mathcal{N}$,
  \eqref{hyp:h+r} is satisfied.
\end{example}

\subsubsection{Statement of Theorem \ref{thm:th5}.}
\label{subsection:th5}

For $t\ge 0$, set $m(t)=\int_0^t \gamma_s ds$.  In Theorem
\ref{thm:th5}, we also have made the following assumptions on $\gamma$
and $\alpha$:
\begin{itemize}
\item $\int_0^\infty \gamma_s ds = \infty$;
\item $\lambda:=\limsup_{t\to\infty}\frac{\log(\alpha(t))}{m(t)} <0$;
\item $\liminf_{t\to\infty} \frac{\log(\alpha(t))}{m(t)}>\beta (1+\nu)$,
\end{itemize}
where $\beta=\sup\{\lambda,\mu\}$, with $\mu$ introduced in subsection
\ref{sec:Ft-assumption}.

Let us recall the statement of Theorem \ref{thm:th5}: Suppose that the
hypotheses described above in this section are satisfied and suppose
that a.s. on $\Gamma$, conditions \eqref{hyp:<M>2} and
\eqref{hyp:calpha} are satisfied by $M$ and $R$, condition
\eqref{hyp:<M>1} is satisfied by $M^+$ and that condition
\eqref{hyp:jumps} is satisfied by $Y^+_t:=(P^{-1}X_t)^+$.  Then it holds
that $\mathbb{P}[\Gamma\cap\{\lim_{t\to\infty}X_t=0\}]=0$.

\subsubsection{Change of time and $\lambda$-APT}
For $t\ge 0$, set $m(t)=\int_0^t \gamma_s ds$, $X^m_t=X_{m^{-1}(t)}$,
$M^m_t=M_{m^{-1}(t)}$ and $R^m_t=R_{m^{-1}(t)}$, where
$m^{-1}(t)=\inf\{s:\; m(s)>t\}$. Then $X^m$ satisfies
\begin{equation}
  X^m_t=X^m_0 + \int_0^t f(X^m_{s-}) ds + M^m_t +R^m_t.
\end{equation}
Set \(V_{t}=V(R,(0,t])\) and \(V^{m}_{t}=V(R^{m},(0,t])\). Then
$\langle M^m\rangle_t=\langle M\rangle_{m^{-1}(t)}$ and
$V^m_t=V_{m^{-1}(t)}$. Set $\alpha_m(t)=\alpha(m^{-1}(t))$.

In many applications, $\alpha(t)=\frac{1}{\sqrt{t}}$, $m(t)=\log(t)$,
$m^{-1}(t)=e^{t}$ and so $\alpha_m(t)=e^{-t/2}$.

\begin{lemma}\label{lem:laptg}
  a.s. on $\Gamma$, $X^m$ is a $\lambda$-asymptotic pseudotrajectory (or
  a $\lambda$-APT), i.e. for all $T>0$,
  \begin{equation}
    \limsup_{t\to\infty} \frac{1}{t} \log \left( \sup_{0\le h\le T} d\big(X^m_{t+h},\varphi_h(X^m_t)\big)\right) \le \lambda.
  \end{equation}
\end{lemma}
\begin{proof}
  We essentially follow the proof of Proposition 8.3 in \cite{Benaim99}.

  For $(t_0,k)\in \mathbb{N}^2$, let $\Gamma_{t_0,k}$ be the set of all
  $\omega\in \Gamma$ such that for all $t\ge t_0$,
  $\langle M\rangle_{t,\infty}\le k^2\alpha^2_t$. Then
  $\Gamma=\cup_{t_0,k}\Gamma_{t_0,k}$ and it suffices to prove the lemma
  with $\Gamma$ replaced by $\Gamma_{t_0,k}$.  So we now fix
  $(t_0,k)\in \mathbb{N}^2$. Let
  $\sigma:=\inf\{ t\ge t_0:\; \sup_{s\in [t_0,t]} \frac{\langle
    M\rangle_{s,t}}{\alpha^2_s}>k^2\}$.  Then $\sigma$ is a stopping
  time, $\sigma=\infty$ on $\Gamma_{t_0,k}$, and for $t>t_0$,
  $\langle M^\sigma \rangle_{t,\infty}\le k^2\alpha^2_t$.  Set
  $I^\sigma_t:=\sup_{u\ge
    t}\|M^\sigma_{t,u}\|$.

  By Doob's inequality, it holds that for $t>m(t_0)$ and $\delta>0$,
    $$\mathbb{P}\left[I^{\sigma}_{m^{-1}(t)} > e^{-\delta t}\right]\le e^{2\delta t} \mathbb{E}[\langle M^\sigma\rangle_{m^{-1}(t),\infty}]\le e^{2\delta t} k^2 \alpha^2_m(t).$$
    Let $\varepsilon>0$ and $\delta>0$ be such that
    $\delta+\varepsilon<-\lambda$.  Since
    $\lambda=\limsup_{t\to\infty} \frac{\log(\alpha_m(t))}{t} <0$ it
    holds that
    $\alpha_m(t)=o\big(e^{(\lambda+\varepsilon)t}\big)$. Therefore,
    $\mathbb{P}\left[I^{\sigma}_{m^{-1}(t)} > e^{-\delta
        t}\right]=o\big(e^{2(\delta +\lambda+\varepsilon) t}\big)$. This
    easily implies (using Borel-Cantelli Lemma) that
    $I^{\sigma}_{m^{-1}(t)}=0(e^{-\delta t})$.

    For $(t,T)\in \mathbb{R}_+^2$, set
    \begin{equation}
      \Delta(t,T)=\sup_{h\in [0,T]} \|M^m_{t,t+h}+R^m_{t,t+h}\|. 
    \end{equation}
    On $\Gamma_{t_0,k}$, $\sigma=\infty$ and
    $\Delta(t,T)\le I^\sigma_{m^{-1}(t)} + V^m_{t,\infty}$.  Therefore,
    for all fixed $T$, a.s. on $\Gamma_{t_0,k}$, as $t\to\infty$,
    $\Delta(t,T)= 0(e^{-\delta t}) + o(\alpha_m(t))= 0(e^{-\delta t})$.

    This proves that a.s. on $\Gamma_{t_0,k}$,
    $\limsup_{t\to\infty} \frac{1}{t} \log \Delta(t,T)\le -\delta$. This
    holds for all $-\delta>\lambda$ and so a.s. on $\Gamma_{t_0,k}$,
    $$\limsup_{t\to\infty} \frac{1}{t} \log \Delta(t,T)\le \lambda.$$
    As noticed in \cite{Benaim99}, this concludes the proof of this
    lemma.
  \end{proof}

  Following the proof of Lemma 8.7 in \cite{Benaim99} we prove the
  following lemma (note that Lemma 8.7 in \cite{Benaim99} cannot be
  directly applied since in our setting: $K$ is not positively invariant
  and $K$ does not attracts $\mathcal{N}^*$ exponentially fast in the
  sense of \cite{Benaim99}):
  \begin{lemma}\label{lem:beta-attract}
    a.s. on $\Gamma\cap\{\lim_{t\to\infty}X_t=0\}$,
    $$\limsup_{t\to\infty}\frac{1}{t}\log\big( d(X^m_t,K) \big) \le \beta.$$
  \end{lemma}
  \begin{proof}
    The proof of this lemma is almost identical to the one of Lemma 8.7
    in \cite{Benaim99}. For completeness, we give a proof of this lemma.
    Let us fix $\varepsilon\in (0,\beta)$ and $T>0$, let
    $\mathcal{N}_\varepsilon$ be the neighborhood introduced in Lemma
    \ref{lem:muattract}. Then there is a neighborhood
    $\mathcal{N}_{\varepsilon,T}\subset \mathcal{N}_\varepsilon$ such that
    $t_\varepsilon(x)\ge T$ for all $x\in \mathcal{N}_{\varepsilon,T}$ (to
    construct such neighborhood it suffices to use that
    $f(x)=O(\|x\|)$). We then have that for all
    $x\in \mathcal{N}_{\varepsilon,T}$,
    \[
      d(\varphi_T(x),K)\le e^{(\mu+\varepsilon)T} d(x,K).
    \]
    Thus a.s. on $\Gamma\cap\{\lim_{t\to\infty}X_t=0\}$, there is $t_0$
    such that for all $t\ge t_0$, $X^m_t\in \mathcal{N}_{\varepsilon,T}$ and
    \[
      d(X^m_{t+T},K)\le
      d(X^m_{t+T},\varphi_T(X^m_t))+d(\varphi_T(X^m_t),K) \le
      e^{(\lambda+\varepsilon)t} + e^{(\mu+\varepsilon)T} d(X^m_t,K).
    \]
    Let $v_k=d(X^m_{kT},K)$, $\rho=e^{(\mu+\varepsilon)T}$ and
    $k_0=[t_0/T]+1$. Then $v_{k+1}\le \rho^k+\rho v_k$ for $k\ge k_0$.
    Hence,
    \[
      v_{k_0+m} \le \rho^m(m\rho^{k_0-1}+v_{k_0})
    \]
    for $m\ge 1$. It follows that
    \[
      \limsup_{k\to\infty} \frac{\log (v_k)}{kT} \le \frac{\log
        (\rho)}{T}\le \beta +\varepsilon.
    \]
    Also for $t\in [kT,(k+1)T]$ and $k\ge k_0$,
    \begin{eqnarray*}
      d(X^m_t,K) &\le& d(\varphi_{t-kT}(X^m_{kT}),X^m_t) + d(\varphi_{t-kT}(X^m_{kT}),K)\\ 
                 &\le& e^{(\lambda+\varepsilon)kT} + e^{(\mu+\varepsilon)(t-kT)}d(X^m_{kT},K)\\
                 &\le& e^{(\lambda+\varepsilon)kT} + v_k.
    \end{eqnarray*}
    Thus
    \[
      \limsup_{t\to\infty} \frac{\log (d(X^m_t,K))}{t} \le \beta
      +\varepsilon.
    \]
    And since $\varepsilon$ is arbitrary we get the desired result.
  \end{proof}

  Lemma \ref{lem:beta-attract} implies that a.s. on
  $\Gamma\cap\{\lim_{t\to\infty}X_t=0\}$, for all $\varepsilon>0$,
  $d(X_t,K)=O\big(e^{(\beta+\varepsilon)m(t)}\big)$. 

\subsubsection{Proof of Theorem \ref{thm:th5}.}
The process $Y^+(=(P^{-1}X)^+)$ satisfies
\begin{equation}
  Y^+_t = Y^+_0 +\int_0^t \gamma_s\mathfrak{h}(Y^+_{s-}) ds +  M^+_t + R^{Y^+}_t
\end{equation}
with $M^+_t:=(P^{-1}M^m_t)^+$ and
$R^{Y^+}_t:=(P^{-1}R^m_t)^++\int_0^t
\gamma_s\left(h^+\big(Y_{s-}\big)-\mathfrak{h}(Y^+_{s-})\right) ds.$

Since $\lim_{t\to\infty}Y^+_t=0$ on the event
$\Gamma\cap\{\lim_{t\to\infty}X_t=0\}$, if Theorem \ref{thm:th4} can be
applied to $Y^+$ and to the event
$\Gamma\cap\{\lim_{t\to\infty}X_t=0\}$, then one has shown that
$\mathbb{P}(\Gamma\cap\{\lim_{t\to\infty}X_t=0\})=0$ and Theorem \ref{thm:th5}
is proved.  We have assumed in Theorem \ref{thm:th5} that condition
\eqref{hyp:jumps} is satisfied by $Y^+$ and that conditions
\eqref{hyp:<M>1} and \eqref{hyp:<M>2} are satisfied a.s. on
$\Gamma\cap\{\lim_{t\to\infty}X_t=0\}$. Therefore, $0$ being a repulsive
equilibrium (i.e. \textit{(i)} of Theorem \ref{thm:th4} is satisfied),
it remains to check that $R^{Y^+}$ satisfies \eqref{hyp:calpha} a.s. on
$\Gamma\cap\{\lim_{t\to\infty}X_t=0\}$.
\begin{lemma}
  $R^{Y^+}$ satisfies \eqref{hyp:calpha} a.s. on
  $\Gamma\cap\{\lim_{t\to\infty}X_t=0\}$.
\end{lemma}
\begin{proof}
  Lemma \ref{lem:beta-attract} implies that a.s. on
  $\Gamma\cap\{\lim_{t\to\infty}X_t=0\}$, for all $\varepsilon>0$,
  $d(X_t,K)=\|Y^-_t-g(Y^+_t)\|=O\big(e^{(\beta+\varepsilon)m(t)}\big)$.
  Assumption \eqref{hyp:h+r} then implies that a.s on
  $\Gamma\cap\{\lim_{t\to\infty}X_t=0\}$, for all $\varepsilon>0$,
  \begin{align}
    h^+\big(Y_t\big)
    &\;=\mathfrak{h}\big(Y^+_t\big) + O(\|Y^-_t-g\big(Y^+_t\big)\|^{1+\nu})\\
    &\;=\mathfrak{h}\big(Y^+_t\big) + O(e^{(\beta+\varepsilon)(1+\nu)m(t)})
  \end{align}
  This implies that a.s on $\Gamma\cap\{\lim_{t\to\infty}X_t=0\}$,
  \begin{align*}
    \int_t^\infty \gamma_s\left\|h^+\big(Y_{s-}\big)-\mathfrak{h}(Y^+_{s-})\right\| ds
    &\;= O\left( \int_t^\infty \gamma_se^{(\beta+\varepsilon)(1+\nu)m(s)} ds\right)\\
    &\;= O\left( \int_{m(t)}^\infty e^{(\beta+\varepsilon)(1+\nu)u} du\right)\\
    &\;= O(e^{(\beta+\varepsilon)(1+\nu)m(t)}).
  \end{align*}
  Since
  $\liminf_{t\to\infty}\frac{\log\big(\alpha(t)\big)}{m(t)}>\beta
  (1+\nu)$, taking $\varepsilon$ sufficiently small, we obtain
  \eqref{hyp:calpha}.
\end{proof}

\section{Proofs of the non-convergence theorems for discrete-time
  processes}
\label{sec:discrete}

Let $(X_n)_{n\in \mathbb{N}}$ be a random sequence in $\mathbb{R}^d$,
adapted to a filtration $(\mathcal{F}_n)_{n\in\mathbb{N}}$, such that
\begin{equation}\label{eq:Xn0}
  X_{n+1}-X_n=\gamma_{n+1}G_n + c_{n+1}(\varepsilon_{n+1}+r_{n+1}),
\end{equation}
where $(\gamma_n)_{n\ge 0}$ and $(c_n)_{n\ge 0}$ are non-negative
deterministic sequences, $(\varepsilon_n)_{n\ge 0}$, $(r_n)_{n\ge 0}$
and $(G_n)_{n\ge 0}$ are adapted sequences such that for all $n\ge 0$,
$\mathbb{E}[\varepsilon_{n+1}|\mathcal{F}_n]=0$ and
$\mathbb{E}[\|\varepsilon_{n+1}\|^2|\mathcal{F}_n]<\infty$.  We will assume that
$c_n\ne 0$ infinitely often and that $\sum_n c_n^2<\infty$.

In
\textcolor{black}{Theorem~\ref{thm:th2n}},~\textcolor{black}{Proposition~\ref{thm:th22n}},
\textcolor{black}{Corollary~\ref{thm:th3bd} and
  Theorem~\ref{thm:th4d}}, it is supposed that for some event
\(\Gamma\) and some constant \(a>2\), a.s. on \(\Gamma\),
\(\limsup\mathbb{E}[\|\varepsilon_{n}\|^a|\mathcal{F}_{n-1}]<\infty\). We
thus have that a.s. \(\Gamma=\cup_{n_{0},k}\Gamma_{n_{0},k}\), where
\(\Gamma_{n_{0},k}\) is the set of all \(\omega\in\Gamma\) such that
for all \(n\geq{n_{0}}\),
\(\mathbb{E}[\|\varepsilon_{n}\|^a|\mathcal{F}_{n-1}]\leq{k^{a}}\). It
is then straightforward to check that it suffices to prove these
theorems with \(\Gamma\) replaced by \(\Gamma_{n_{0},k}\). In the
following we will therefore suppose from now on in this section that
\(\Gamma=\Gamma_{n_{0},k}\) for some \((n_{0},k)\in\mathbb{N}^{2}.\)

For $t\ge 0$ and $n=\lfloor t\rfloor$, set
$\mathcal{F}_t=\mathcal{F}_{n}$. Then $(\mathcal{F}_t)_{t\ge 0}$ is a
complete right-continuous filtration. Let us define the càdlàg processes
\(F\), $M$, $R$ and $X$ such that for all $n\ge 0$ and $t\in [n,n+1)$,
\begin{equation}
  F_t=\gamma_{n+1}G_n,\qquad M_t=\sum_{k=1}^n c_{k}\varepsilon_{k}, \qquad R_t=\sum_{k=1}^n c_{k} r_{k}
\end{equation}
and
\begin{align}
    & X_t=X_n+\int_n^t F_s \,\mathrm{d}s.
\end{align}
Then, $X$, $F$, $M$ and $R$ are adapted to $(\mathcal{F}_t)_{t\ge 0}$
and satisfy \eqref{eq:Z} for all $t\ge 0$.  Note also that, with respect
to the filtration $(\mathcal{F}_t)_{t\ge 0}$, $F$ is progressively
measurable and $M$ is a locally square integrable martingale.  Moreover,
for $t\ge 0$,
\begin{align}
  & [M]_t=\sum_{n\le t} c^2_n\|\varepsilon_{n}\|^2\\
  &\langle M\rangle_t = \sum_{n\le t} c^2_n\mathbb{E}[\|\varepsilon_{n}\|^2|\mathcal{F}_{n-1}]\\
  & V(R,(0,t])=\sum_{n\le t} c_{n}\|r_{n}\|.
\end{align}
We also have that $\Delta X_t=0$ if $t\not\in\mathbb{N}^*$ and that
$\Delta X_t=c_{n}(\varepsilon_{n}+r_{n})$ if $t=n\in \mathbb{N}^*$.

Let us also set for $t\ge 0$, $\alpha_t=\sqrt{\sum_{n> t} c_n^2}$ so
that $\alpha:[0,\infty)\to (0,\infty)$ is a non-increasing function, and
\(E_{t}=k^{2}\alpha^{2}_{t}\) if \(t\geq{n_{0}}\) and \(E_{t}=\infty\)
if \(t<n_{0}\). Then it is an easy exercice using Lemma~\ref{lem:E_t}
that a.s. on \(\Gamma\), for all stopping time \(S\),
\begin{equation*}
  \frac{\mathbb{E}[\|\Delta M_S\|^21_{S<\infty}|\mathcal{F}_t]}{\mathbb{P}[S<\infty|\mathcal{F}_t]^b}\leq{E_{t}}
\end{equation*}
and that a.s. on \(\Gamma\), \eqref{hyp:jumps} is satisfied.

Let us finally set \(\tau(t)=t+1\) for all \(t\geq{0}\).

\begin{proof}[Proof of Theorem \ref{thm:th2n}]
  In Theorem \ref{thm:th2n}, we also suppose that a.s. on \(\Gamma\),
  \begin{equation*}
    0<\liminf \mathbb{E}[\|\varepsilon_{n}\|^2|\mathcal{F}_{n-1}] \quad\hbox{ and }\quad  \sum_n \|r_n\|^2 < \infty
  \end{equation*}
  and that there is a finite random variable $T_0$ and $\rho>0$ such
  that a.s. on $\Gamma$, for all $n\ge T_0$,
  \begin{align} \label{hyp:Fn}
    \langle X_{n}, G_n\rangle 1_{\{\|X_{n}\|<2\rho\}} \geq 0. 
  \end{align}

  It is straightforward to check that a.s. on $\Gamma$,
  $\lim_{n\to\infty}c_{n+1}(\varepsilon_{n+1}+r_{n+1})=0$ and therefore
  that a.s. on $\Gamma_0:=\Gamma\cap\{\lim_{n\to\infty}X_n=0\}$,
  $\lim_{n\to\infty}
  \gamma_{n+1}G_n=\lim_{n\to\infty}(X_{n+1}-X_n)=0$. So we can choose
  the random time $T_0$ such that a.s. on $\Gamma_0$, for all
  $n\ge T_0$,
  \begin{align} \label{hyp:Fnb} \langle X_{n}, G_n\rangle
    1_{\{\|X_{n}\|<2\rho\}} \ge 0 \qquad \hbox{and} \qquad
    \gamma_{n+1}\|G_n\|<\rho.
  \end{align}

  In order to apply Theorem \ref{thm:th2} we now verify that a.s. on
  $\Gamma_0$ for all $t\ge T_0$, \eqref{hyp:F1}, \eqref{hyp:<M>1},
  \eqref{hyp:<M>2}, \eqref{hyp:calpha} and \eqref{hyp:jumps} are
  satisfied. We have already checked that \eqref{hyp:jumps} holds
  a.s. on \(\Gamma\). 

  Let us check that condition \eqref{hyp:F1} is satisfied a.s. on
  \(\Gamma_{0}\). If $t\in (n,n+1)$, we have
  $X_{t-}=X_n+(t-n)\gamma_{n+1}G_n$ and $F_t=\gamma_{n+1}G_n$, and so
  \eqref{hyp:Fnb} implies that a.s. on $\Gamma_0$ for almost all $t\ge T_0$,
  $\langle X_{t-},F_t\rangle 1_{\{\|X_{t-}\|< \rho\}}\ge 0$ (using that if
  \(\|X_{t-}\|<\rho\), then
  \(\|X_{n}\|\leq{}\|X_{t-}\|+\gamma_{n+1}\|G_{n}\|<2\rho\)),
  i.e. condition \eqref{hyp:F1} is satisfied.

  For $t\in [n-1,n)$,
  \begin{align*}
    &\langle{M}\rangle_{t,t+1}=\langle M \rangle_{t+1}-\langle M \rangle_t=c^2_n\mathbb{E}[\|\varepsilon_{n}\|^2|\mathcal{F}_{n-1}]\\
    &\alpha^2_t-\alpha^2_{t+1}=c_n^2
  \end{align*}
  and so
  $\liminf \frac{\langle{M}\rangle_{t,t+1}}{\alpha^2_t-\alpha^2_{t+1}}
  = \liminf \mathbb{E}[\|\varepsilon_{n}\|^2|\mathcal{F}_{n-1}] >0$ a.s. on
  $\Gamma_0$.  This proves that, a.s. on $\Gamma_0$, \eqref{hyp:<M>1} is
  satisfied with $\tau(t)=t+1$.  We also have that
  \begin{align*}
    \frac{\langle M \rangle_{t,\infty}}{\alpha^2_t} =
    \frac{\sum_{n>t}c^2_n\mathbb{E}[\|\varepsilon_{n}\|^2|\mathcal{F}_{n-1}]}{\sum_{n>t}c_n^2}
  \end{align*}
  and therefore
  \begin{align*}
    \limsup \frac{\langle M \rangle_{t,\infty}}{\alpha^2_t}
    &\le \limsup \mathbb{E}[\|\varepsilon_{n}\|^2|\mathcal{F}_{n-1}]\\
    &\le \limsup \mathbb{E}[\|\varepsilon_{n}\|^a|\mathcal{F}_{n-1}]^{2/a} < \infty.
  \end{align*}
  This proves \eqref{hyp:<M>2}.

  Applying Cauchy-Schwartz inequality,
  \begin{align*}
    V(R,(t,\infty)) = \sum_{n>t} c_n\|r_{n}\|\le  \alpha_t \sqrt{\sum_{n>t}\|r_{n}\|^2}.
  \end{align*}
  This implies condition \eqref{hyp:calpha}, a.s. on $\Gamma_0$.

  Theorem~\ref{thm:th2} can therefore be applied and this proves Theorem~\ref{thm:th2n}.
\end{proof}

\begin{proof}[Proof of \textcolor{black}{Proposition~\ref{thm:th22n}}]
  In \textcolor{black}{Proposition~\ref{thm:th22n}}, it is supposed that
  \(\gamma_{n}=O(c_{n})\) ant that there are a random variable
  \(T_{0}\), \(\beta>0\), \(\rho>0\) and \(\nu>0\) for
  which~\eqref{hyp:G22n} and~\eqref{hyp:r22n} are satisfied a.s. on
  \(\Gamma\) (for convenience, we replace in this proof \(\rho\) by
  \(2\rho\) in~\eqref{hyp:G22n}).

  It is straightforward to check that a.s. on
  $\Gamma$, $\lim_{n\to\infty}c_{n+1}(\varepsilon_{n+1}+r_{n+1})=0$ and
  therefore that a.s. on
  $\Gamma_0:=\Gamma\cap\{\lim_{n\to\infty}X_n=0\}$,
  $\lim_{n\to\infty}
  \gamma_{n+1}G_n=\lim_{n\to\infty}(X_{n+1}-X_n)=0$. So we can choose
  the random time $T_0$ such that a.s. on $\Gamma_0$, for all
  $n\ge T_0$, \(\beta\gamma_{n+1}\leq{1}\) and
  \begin{align} \label{hyp:G22nb} \langle X_{n}, G_n\rangle
    1_{\{\|X_{n}\|<2\rho\}} \ge \beta\|X_{n}\|^{2}  \qquad \hbox{and} \qquad
    \gamma_{n+1}\|G_n\|<\rho.
  \end{align}

  Set \(\gamma_{t}=\frac{\beta}{2}\gamma_{n+1}\) if \(t\in(n,n+1]\)
  and \(\alpha_{t}\) as it is defined in the proof of
  Theorem~\ref{thm:th2n}. Then (using that \(\gamma_{n}=O(c_{n})\)),
  \(\int_{t}^{\infty}\gamma^{2}_{s}\mathrm{d}s=O(\alpha_{t}^{2})\).
  The definition of the processes \(X\), \(F\), \(M\) and \(R\) is
  slightly modified: we now set
  \(R_{t}=R'_{t}+\int_{0}^{t}r''_{s}\,\mathrm{d}s\) where, for
  \(t\in[n,n+1)\), \(R'_t=\sum_{k=1}^n c_{k} r'_{k}\) and
  \(r''_{t}=c_{n+1}r''_{n+1}\). In order to apply
  \textcolor{black}{Proposition~\ref{thm:th22}}, we now verify that
  a.s. on $\Gamma_0$ for all $t\ge T_0$, \eqref{hyp:<M>1},
  \eqref{hyp:<M>2}, \eqref{hyp:jumps}, \eqref{hyp:F12},
  \eqref{hyp:calpha2} and~\eqref{hyp:calpha3} are satisfied. We have
  already checked that \eqref{hyp:jumps} holds a.s. on \(\Gamma\).
  Conditions~\eqref{hyp:<M>1} and~\eqref{hyp:<M>2} can be checked as
  in the proof of Theorem~\ref{thm:th2n}.

  Let us check that condition \eqref{hyp:F12} is satisfied:
  a.s. on \(\Gamma_{0}\) for all $t\in (n,n+1)$ such that
  \(t\geq{}T_{0}\) and such that \(\|X_{t}\|<\rho\), it holds that
  \(\|X_{n}\|<2\rho\) and since (using in the last inequality  that
  \(\beta\gamma_{n+1}\leq{1}\) and that \(\|X_{t-}\|^{2}\leq{}2(\|X_{n}\|^{2}+\gamma^{2}_{n+1}\|G_{n}\|^{2})\))
  \begin{align*}
    \langle{X_{t-}},F_{t}\rangle
    &= \gamma_{n+1}\langle{X_{n}},G_{n}\rangle +
      \gamma_{n+1}^{2}\|G_{n}\|^{2}\\
    &\geq \beta\gamma_{n+1}\|X_{n}\|^{2} +
      \gamma_{n+1}^{2}\|G_{n}\|^{2}\\
    &\geq \frac{\beta}{2}\gamma_{n+1}\|X_{t-}\|^{2}.
  \end{align*}
  condition \eqref{hyp:F12} is therefore satisfied a.s. on
  \(\Gamma_{0}\). Condition~\eqref{hyp:calpha2} can be verified exactly
  as is verified condition~\eqref{hyp:calpha} in the proof of
  Theorem~\ref{thm:th2n}. The last condition to be verified is
  condition~\eqref{hyp:calpha3}. Condition~\eqref{hyp:r22n} implies that
  for \(t\in[n-1,n)\),
  \begin{equation*}
    r''_{t}=c_{n}r''_{n}=O(\gamma_{n}^{1+\nu})=O(\gamma_{t}^{1+\nu}).
  \end{equation*}
  \textcolor{black}{Proposition~\ref{thm:th22}} can therefore be applied
  and this proves \textcolor{black}{Proposition~\ref{thm:th22n}}.
\end{proof}

\begin{proof}[Proof of \textcolor{black}{Corollary~\ref{thm:th3bd}}]
  This corollary can easily be proved using Theorem~\ref{thm:th2n} as is
  proved \textcolor{black}{Corollary~\ref{thm:th3b}} using Theorem~\ref{thm:th2}. Let us show
  how to prove \textcolor{black}{Corollary~\ref{thm:th3bd} as a corollary of
  Corollary~\ref{thm:th3b}}.
 
  In \textcolor{black}{Corollary~\ref{thm:th3bd}}, it is supposed that \(H_{n}\) is an
  adapted sequence of random matrices converging a.s. on \(\Gamma\)
  towards a repulsive matrix \(H\) and that there is a random variable
  \(T_{0}\) such that a.s. on \(\Gamma\), \(G_{n}=H_{n}X_{n}\) for all
  \(n\geq{T_{0}}\) and that~\eqref{hyp:rn} are satisfied a.s. on
  \(\Gamma\).

  In order to apply \textcolor{black}{Corollary~\ref{thm:th3b}}, we set
  \(\gamma_{t}=\gamma_{n+1}\) for all \(t\in(n,n+1]\). Note that we can
  choose \(T_{0}\) such that a.s. on \(\Gamma\), for all \(t\in(n,n+1)\)
  and \(n\geq{T_{0}}\), the matrix \(I+(t-n)\gamma_{n+1}H_{n}\) is
  invertible. Let us then define a random matrix \(H_{t}\)
  such that a.s. on \(\Gamma\),
  \(H_{t}=H_{n}(I+(t-n)\gamma_{n+1}H_{n})^{-1}\) for all \(t\in(n,n+1)\)
  and \(n\geq{T_{0}}\), so that
  \(\gamma_{t}H_{t}X_{t-}=\gamma_{t}H_{t}(X_{n}+(t-n)\gamma_{n+1}H_{n}X_{n})=\gamma_{n+1}H_{n}X_{n}=F_{t}\).
  
  Then a.s. on \(\Gamma\), \(\lim_{t\to\infty}H_{t}=H\).
  Condition~\eqref{hyp:jumps} has already been checked and
  conditions~\eqref{hyp:<M>1},~\eqref{hyp:<M>2} and~\eqref{hyp:calpha}
  can be verified as in the proof of Theorem~\ref{thm:th2n}.
  
  \textcolor{black}{Corollary~\ref{thm:th3b} can therefore be applied and this proves
  Corollary~\ref{thm:th3bd}.}
\end{proof}

\begin{proof}[Proof of Theorem \ref{thm:th4d}]
  In Theorem \ref{thm:th4d}, $x^*\in \mathbb{R}^d$ is an unstable
  equilibrium of a vector field $f:\mathbb{R}^d\to\mathbb{R}^d$ such
  that $f$ is $C^{1}$ in a convex neighborhood $\mathcal{N}^*$ of $x^*$.
  For all $n\ge 0$, $G_n=f(X_n)$ and it is also supposed that
  a.s. on $\Gamma$,
  \begin{align*}
    0<\liminf \mathbb{E}[\|\varepsilon^+_{n}\|^2|\mathcal{F}_{n-1}] \quad\hbox{ and }\quad \sum_n \|r_n\|^2 < \infty.
  \end{align*} 
  It is set $\alpha(t)=\sqrt{\sum_{n> t} c_n^2}$ and it is also supposed that one of the
  two following conditions is satisfied
  \begin{enumerate}[\itshape (i)]
  \item $x^*$ is repulsive;
  \item $\sum_{n>t}\gamma_{n}^2=O(\alpha(t))$ (which is satisfied if
    $\gamma_n=0(c_n)$) and there is $\nu\in (0,1]$ such that
    $f\in C^{1+\nu}(\mathcal{N}^*)$ and that
    $\sum_{n>t} c_n^{1+\nu} \|\varepsilon^-_n\|^{1+\nu}=o(\alpha_t)$ on
    $\Gamma\cap\{\lim X_n=x^*\}$.
  \end{enumerate}

  In order to apply Theorem~\ref{thm:th4}, we set
  \(\gamma_{t}=\gamma_{\lceil{t}\rceil}\),
  \(M^{+}_{t}=\sum_{k\leq{t}}c_{k}\varepsilon^{+}_{k}\) and
  \(M^{-}_{t}=\sum_{k\leq{t}}c_{k}\varepsilon^{-}_{k}\). Note that for
  \(t\in(n,n+1)\), \(F_{t}=\gamma_{n+1}f(X_{n})\). Setting
  \(\tilde{F}_{t}=\gamma_{t}f(X_{t})\) and
  \(\tilde{R}_{t}=R_{t}+\int_{0}^{t}(F_{s}-\tilde{F}_{s})\,\mathrm{d}s\),
  we get that for all \(t\geq{0}\),
  \begin{equation*}
    X_{t}-X_{0}=\int_{0}^{t}\tilde{F}_{s}\,\mathrm{d}s+M_{t}+\tilde{R}_{t}.
  \end{equation*}

  We now apply Theorem~\ref{thm:th4} with \(F\) and \(R\) replaced by
  \(\tilde{F}\) and \(\tilde{R}\).  As in the proof of
  Theorem~\ref{thm:th2n}, conditions~\eqref{hyp:<M>1}
  and~\eqref{hyp:<M>2} can be verified for
  \(M^{+}\). Condition~\eqref{hyp:jumps} has already been verified. Let
  us verify that condition~\eqref{hyp:calpha} is satisfied by
  \(\tilde{R}\) a.s. on \(\Gamma\cap\{\lim_{n\to\infty}X_{n}=x^{*}\}\). Note that a.s. on
  \(\Gamma\cap\{\lim_{n\to\infty}X_{n}=x^{*}\}\), setting
  \(n=\lfloor{t\rfloor}\),
  \(\|F_{t}-\tilde{F}_{t}\|=\gamma_{n+1}\|f(X_{n})-f(X_{t})\|=o(\gamma_{n+1}^{2})\)
  (using that \(\|X_{t}-X_{n}\|\leq\gamma_{n+1}\|f(X_{n})\|\)).
  
  We therefore have that a.s. on
  \(\Gamma\cap\{\lim_{n\to\infty}X_{t}=x^{*}\}\)
  \begin{align*}
    V(\tilde{R},(t,\infty))
    &\leq \sum_{k>t}c_{k}\|r_{k}\| +
      \int_{t}^{\infty}\|F_{s}-\tilde{F}_{s}\|\,\mathrm{d}s \\
    &\leq \sqrt{\sum_{k>t}c_{k}^{2}}\times\sqrt{\sum_{k>t}\|r_{k}\|^{2}} +
      o\left(\sum_{k>t}\gamma_{k}^{2}\right) \\
    &= o(\alpha_{t})
  \end{align*}
  i.e. condition~\eqref{hyp:calpha} is satisfied by
  \(\tilde{R}\) a.s. on \(\Gamma\cap\{\lim_{n\to\infty}X_{n}=x^{*}\}\).
  
  We now verify that a.s. on \(\Gamma\), 
  \(\langle{M^{-}}\rangle_{t,\tau(t)}=O(\langle{M^{+}}\rangle_{t,\tau(t)})\). This
  easily follows from the fact that (setting \(n=\lfloor{t}\rfloor\))
  a.s. on \(\Gamma\)
  \begin{equation*}
    \langle{M^{-}}\rangle_{t,t+1}=c_{n+1}^{2}\mathbb{E}[\|\varepsilon^{-}_{n+1}\|^{2}|\mathcal{F}_{n}]
    \leq{}c_{n+1}^{2}\mathbb{E}[\|\varepsilon_{n+1}\|^{2}|\mathcal{F}_{n}] =O(c_{n}^{2})=O(\langle{M^{+}}\rangle_{t,t+1})
  \end{equation*}
  since
  \(\langle{M^{+}}\rangle_{t,t+1}=c_{n+1}^{2}\mathbb{E}[\|\varepsilon^{+}_{n+1}\|^{2}|\mathcal{F}_{n}]\)
  and since a.s. on \(\Gamma\),
  \(\liminf\mathbb{E}[\|\varepsilon^{+}_{n}|\mathcal{F}_{n-1}]>0\).

  Finally, we can conclude when \textit{(i)} of Theorem~\ref{thm:th4d}
  is satisfied by using the case \textit{(i)} of
  Theorem~\ref{thm:th4}. When \textit{(ii)} of theorem~\ref{thm:th4d} is
  satisfied, to use the case \textit{(ii)} of
  Theorem~\ref{thm:th4}, it remains to check that \(M\) is purely
  discontinuous (which is straightforward) and that a.s. on \(\Gamma\),
  \(\sum_{s>t}\|\Delta{M^{-}_{s}}\|^{1+\nu}=o(\alpha_{t})\) (which
  follows from the fact that
  \(\sum_{s>t}\|\Delta{M^{-}_{s}}\|^{1+\nu}=\sum_{n>t}c_{n}^{1+\nu}\|{\varepsilon^{-}_{n}}\|^{1+\nu}\)).
\end{proof}

\begin{proof}[Proof of Theorem~\ref{thm:th5d}]
  For this theorem, we have \(G_{n}=f(X_{n})\). As in the proof of
  Theorem~\ref{thm:th4d}, for \(t\geq{0}\), set
  \(\tilde{F}_{t}=\gamma_{t}f(X_{t})\) and
  \(\tilde{R}_{t}=R_{t}+\int_{0}^{t}(F_{s}-\tilde{F}_{})\,\mathrm{d}s\). We
  will apply Theorem~\ref{thm:th5} to the process \(X\), with \(F\) and
  \(R\) replaced by \(\tilde{F}\) and \(\tilde{R}\). It is
  straightforward to check that assumptions (i) and (ii) of
  Theorem~\ref{thm:th5d} imply assumptions (i) and (ii) of
  Theorem~\ref{thm:th5} (with the same constants \(\lambda\) and
  \(\mu\)). To conclude, it remains to check condition (iii) of
  Theorem~\ref{thm:th5}.

  It is straightforward to check that~\eqref{hyp:<M>2} is satisfied
  a.s. on \(\Gamma\), and~ it can be verified that~\eqref{hyp:calpha} is
  satisfied a.s. on \(\Gamma\) exactly as in the proof of
  Theorem~\ref{thm:th4d}. Since
  \(\liminf\frac{\langle{M^{+}}\rangle_{t,t+1}}{\alpha_{t}^{2}-\alpha_{t+1}^{2}}
  = \liminf\mathbb{E}[\|\varepsilon^{+}_{n}\|^{2}|\mathcal{F}_{n-1}]\),
  \eqref{hyp:<M>1} is satisfied by \(M^{+}\) a.s. on \(\Gamma\). The
  fact that~\eqref{hyp:jumps} is satisfied a.s. on \(\Gamma\) (with
  \(M\) replaced by \((P^{-1}M)^{+}\) in~\eqref{eq:DME}) can easily be
  proved using Lemma~\ref{lem:E_t}. Theorem~\ref{thm:th5} can therefore
  be applied and this proves Theorem~\ref{thm:th5d}.
\end{proof}
\section{Examples}
\label{sec:example}
\subsection{Strongly vertex reinforced random walks on complete
  graphs}
\label{sec:VRRW}

We give here a correct proof of Theorem 3.9 in \cite{Benaim13}.  Denote
by $v_n$ the empirical occupation measure at time $n$ of a strongly
reinforced VRRW\footnote{VRRW stands for Vertex Reinforced Random Walk}
on a complete graph with $d$ vertices with reinforcement weight
$\omega(k)=k^ \alpha$, where $\alpha>1$.  Let $v^*\in\mathbb{R}^d$ be
such that $v^*_i\in [0,1]$ and $\sum_i v^*_i=1$.  In \cite{Benaim13}
there is a vector field $f$ defined in a neighborhood $\mathcal{N}^*$ of
$v^*$ by
\begin{equation}
  f_i(v)=-v_i+\frac{v_i^\alpha\sum_j A_{ij}v_j^\alpha}{H(v)}
\end{equation}
where $H(v)=\sum_{i,j} A_{ij} v_i^\alpha v_j^\alpha$ and $A$ is a
symmetric matrix with non-negative entries such that $A_{ij}>0$ if
$i\neq j$ and $\sum_j A_{ij}$ does not depend on $i$.

Theorem 3.9 states that when $v^*$ is an unstable equilibrium for $f$,
then $\mathbb{P}[\lim_{n\to\infty} v_n=v^*]$.  We thus let $v^*$ be an
unstable equilibrium.  In the proof of Theorem 3.9 given in
\cite{Benaim13}[Theorem], there is a process $Z_n$ such that
$\lim_{n\to\infty} \|Z_n-v_n\|=0$ and such that \eqref{eq:Xn} is
satisfied by $Z_n$ with $\gamma_n=c_n=n^{-1}$ and $G_n=f(Z_n)$.  It is
easy to see that if $\mathcal{S}:=\{i:\; v^*_i>0\}$ denotes the support
of $v^*$, then
$K:=\{z\in \mathcal{N}^*:\; z_i=0 \hbox{ if } i\not\in \mathcal{S}\}$ is
the (local) unstable manifold of $f$. Moreover, it can be shown that $K$
attracts a neighborhood of $v^*$ at rate $1$ and it can be checked that
\eqref{hyp:h+r} is satisfied with $\nu=\alpha-1>0$.
In \cite{Benaim13}, it has been shown that $\|r_n\|\le C/n$, that
$\mathbb{E}[\varepsilon_{n+1}|\mathcal{F}_n]=0$ and that on the event
$\{\lim_{n\to\infty} Z_n=v^*\}$,
$$\liminf_{n\to\infty} \left(\mathbb{E}[\|\varepsilon^+_{n+2}\|^2\mathcal{F}_{n+1}]+\mathbb{E}[\|\varepsilon^+_{n+1}\|^2|\mathcal{F}_n]\right)>0.$$ 
Therefore Theorem \ref{thm:th5d} with Remark \ref{rem:excitationk} can
be applied (here $k=2$, $\alpha(t)=\sqrt{\sum_{n\ge t} n^{-2}}$,
$\lambda=\lim_{t\to\infty} \frac{\log(\alpha(t))}{\sum_{k\le t}
  \gamma_k}=-\frac12$, $\mu=-1$, $\beta=-\frac12$, $\nu=\alpha-1$ and we
do have $\beta (1+\nu)=-\frac{1+\nu}{2}<\lambda=-\frac{1}{2}$) and we
have proved that $\mathbb{P}[\lim_{n\to\infty} Z_n=v^*]=0$, which is what is
claimed in Theorem 3.9 in \cite{Benaim13}.

\subsection{Non-backtracking VRRW}
\label{sec:NBVVRW}

We correct here non-convergence theorems stated in \cite{LR18}.  In
Section 3 in \cite{LR18}, Theorem 3.27 states the non-convergence
towards an unstable equilibrium $v^*$ of a vector field $f$ on
$\mathbb{R}^d$, for a sequence of random variable $(v_n)_{n\ge 0}$ in
$\mathbb{R}^d$.  In order to prove this theorem, another sequence of
random variable $(Z_n)_{n\ge 0}$ in $\mathbb{R}^d$ is introduced and is
such that $\lim_{n\to\infty} \|Z_n-v_n\|=0$ and such that it satisfies
\eqref{eq:Xn} with $\gamma_n=c_n=(n+d)^{-n}$ and $G_n=f(Z_n)$.  As in
the previous subsection, it is proved in \cite{LR18} that, for some
$m\ge 0$, on the event $\{\lim_{n\to\infty} v_n=v^*\}$,
$\|r_n\|=O(1/n)$, that $\mathbb{E}[\varepsilon_{n+1}|\mathcal{F}_n]=0$ and that on
the event $\{\lim_{n\to\infty} v_n=v^*\}$,
$$\liminf_{n\to\infty} \mathbb{E}\left[\left.\sum_{q=0}^m\|\varepsilon^+_{n+q+1}\|^2\right|\mathcal{F}_{n}\right]>0.$$
So in order to proof that $\mathbb{P}[\lim_{n\to\infty} v_n=v^*]=0$ (or
equivalently that $\mathbb{P}[\lim_{n\to\infty} Z_n=v^*]=0$) one would like to
be able to apply Theorem \ref{thm:th4d} or Theorem \ref{thm:th5d}
together with remark \ref{rem:excitationk} (taking $k=m+1$).  In
\cite{LR18}, Corollary 3.IV.15 in \cite{Duflo1996} is applied, but as
noticed in the previous sections, the proof of this corollary is
incorrect.

With the hypotheses given in \cite{LR18}[Theorem 3.27], one also only
has that $f$ is Lipschitz in a neighborhood of $v^*$, which is not
sufficient to apply Theorem \ref{thm:th4d} or Theorem \ref{thm:th5d}.

To apply Theorem \ref{thm:th4d}, one would have to assume moreover that
$f$ is $C^1$ in a neighborhood of $v^*$ and $v^*$ is repulsive or that
$f$ is $C^{1+\nu}$ in a neighborhood of $v^*$ for some $\nu\in (0,1]$,
$\sum_{k\ge n} \frac{\|\varepsilon_k^-\|^{1+\nu}}{k^{1+\nu}}=o(n^{-1/2})$ on
$\{\lim_{n\to\infty} v_n=v^*\}$ and $v^*$ is unstable.  Note that in
\cite{LR18}[Section 3.3], it is shown that there is $C<\infty$ such that
for all $n$, $\|\varepsilon_n\|\le C$. So
$\sum_{k\ge n} \frac{\|\varepsilon_k^-\|^{1+\nu}}{k^{1+\nu}}=o(n^{-1/2})$ is
satisfied if $\nu>1/2$.

To apply Theorem \ref{thm:th5d}, one would have to assume moreover that
\begin{itemize}
\item $f$ is $C^1$ in a neighborhood of $v^*$,
\item $v^*$ is an hyperbolic equilibrium, and so that there is $K$ a
  local unstable manifold that attracts a neighborhood of $v^*$ at rate
  $\mu$ for some $\mu<0$,
\item Condition \ref{hyp:h+r} is satisfied for some $\nu>0$.
\end{itemize}
As in the previous subsection, $\lambda=-1/2$, and so Theorem
\ref{thm:th5d} can only be applied if the condition $\mu (1+\nu) <-1/2$
is satisfied.

In \cite{LR18} Theorem 3.27 is applied to prove Theorem 1.2 and the
non-convergence towards an unstable equilibrium for the empirical
occupation measure of a class of non-backtracking VRRWs on complete
graphs (in \cite{LR18}[subsection 4.11.3]) with reinforcement weight
$\omega(k)=k^\alpha$, where $\alpha\ge 1$.  As in the previous
subsection, it can be shown easily that $v^*$ is hyperbolic, that the
local unstable manifold is attracted at rate $\mu=-1$ and that
\eqref{hyp:h+r} is satisfied with $\nu=\alpha-1$.  When $\alpha>1$,
$\nu>0$ and Theorem \ref{thm:th5d} with remark \ref{rem:excitationk} can
be applied.  When $\alpha=1$, $f$ is $C^{1+1}$ in a neighborhood of
$v^*$ and Theorem \ref{thm:th4d} with remark \ref{rem:excitationk} can
be applied.  Therefore Theorem 1.2 remains correct for all
$\alpha\ge 1$.

\noindent \textbf{Acknowledgments:} 
\begin{itemize}
\item This research has been conducted within the Fédération
  Parisienne de Modélisation Mathématique (FP2M)–CNRS FR 2036.
\item   This research has been conducted as part of the project Labex
  MME-DII (ANR11-LBX-0023-01).
\end{itemize}

\printbibliography

\end{document}